\documentclass{amsart}
\usepackage{amsmath,amssymb,amsthm,geometry,graphics,color,mdwlist,mathrsfs}
\usepackage[all]{xy}

\DeclareMathOperator{\Hom}{Hom}
\DeclareMathOperator{\sHom}{\underline{\Hom}}

\DeclareMathOperator{\RHom}{\mathrm{R}\!\Hom}

\DeclareMathOperator{\End}{End}
\DeclareMathOperator{\sEnd}{\underline{\End}}

\def\lotimes{\otimes^\mathrm{L}}

\DeclareMathOperator{\Ext}{Ext}

\DeclareMathOperator{\pd}{proj.dim}
\DeclareMathOperator{\id}{inj.dim}

\DeclareMathOperator{\gd}{gl.dim}
\DeclareMathOperator{\gldim}{gl.dim}

\DeclareMathOperator{\md}{mod}
\renewcommand{\mod}{\md}
\DeclareMathOperator{\Mod}{Mod}
\DeclareMathOperator{\smod}{\underline{\md}}

\DeclareMathOperator{\proj}{proj}

\DeclareMathOperator{\fl}{fl}

\DeclareMathOperator{\qmod}{qmod}
\DeclareMathOperator{\CM}{CM}

\DeclareMathOperator{\coh}{coh}

\DeclareMathOperator{\sg}{sg}

\DeclareMathOperator{\thick}{thick}
\DeclareMathOperator{\per}{per}

\DeclareMathOperator{\add}{add}

\DeclareMathOperator{\Proj}{Proj}

\DeclareMathOperator{\SL}{SL}

\DeclareMathOperator{\diag}{diag}
\DeclareMathOperator{\ch}{char}

\def\seg{\#}

\def\rD{\mathrm{D}}
\def\rK{\mathrm{K}}
\def\rC{\mathrm{C}}

\def\Db{\mathrm{D^b}}
\def\Kb{\mathrm{K^b}}

\def\rb{\mathrm{b}}

\newcommand\recollement[3]{\xymatrix{{#1}\ar[r]&{#2}\ar[r]\ar@/_7pt/[l]\ar@/^7pt/[l]&{#3}\ar@/_7pt/[l]\ar@/^7pt/[l] }}
\newcommand\recollementwithmaps[9]{\xymatrix{{#1}\ar[r]|-{#8}&{#2}\ar[r]|-{#5}\ar@/_7pt/[l]_-{#7}\ar@/^7pt/[l]^-{#9}&{#3}\ar@/_7pt/[l]_-{#4}\ar@/^7pt/[l]^-{#6} }}
\def\A{\mathcal{A}}

\def\C{\mathcal{C}}
\def\D{\mathcal{D}}

\def\M{\mathcal{M}}

\def\cQ{\mathcal{Q}}

\def\T{\mathcal{T}}

\def\rU{\mathrm{U}}
\def\G{\Gamma}
\def\Ga{\Gamma}

\def\La{\Lambda}

\def\Om{\Omega}
\def\a{\alpha}

\def\Z{\mathbb{Z}}

\def\PP{\mathbb{P}}

\def\op{\mathrm{op}}

\def\rsimeq{\rotatebox{-90}{$\simeq$}}
\def\xsimeq{\xrightarrow{\simeq}}
\def\ysimeq{\xleftarrow{\simeq}}

\def\vsubset{\rotatebox{90}{$\subset$}}

\newtheorem{Thm}{Theorem}[section]
\newtheorem{Lem}[Thm]{Lemma}
\newtheorem{Prop}[Thm]{Proposition}
\newtheorem{Cor}[Thm]{Corollary}

\newtheorem{Prop-Def}[Thm]{Proposition-Definition}
\newtheorem{Thm-Def}[Thm]{Theorem-Definition}

\theoremstyle{definition}
\newtheorem{Def}[Thm]{Definition}
\newtheorem{Ex}[Thm]{Example}

\newtheorem{Assum}[Thm]{Assumption}

\theoremstyle{remark}
\newtheorem{Rem}[Thm]{Remark}

\newcounter{step}

\def\disoplus{\displaystyle\bigoplus}

\newcommand\Frac[2]{\displaystyle\frac{#1}{#2}}

\def\Binom{\displaystyle\binom}
\def\Wedge{\textstyle\bigwedge}
\makeatletter

\@addtoreset{equation}{section}
\makeatother

\title{Higher hereditary algebras and Calabi-Yau algebras \linebreak arising from some toric singularities}
\author{Norihiro Hanihara}
\thanks{This work is supported by JSPS KAKENHI Grant Number JP22KJ0737. The author is supported by Kavli Institute for the Physics and Mathematics of the Universe (WPI), The University of Tokyo, as an affiliate member.}
\subjclass[2020]{13C14, 16E35, 16S38, 13D02, 16G10, 14A22, 16G60}
\keywords{Veronese subring, Segre product, $n$-representation infinite algebra, singularity category, cluster category, strict root pair, twisted Calabi-Yau algebra}
\address{Faculty of Mathematics, Kyushu University, 744 Motooka, Nishi-ku, Fukuoka, 819-0395, Japan}
\email{hanihara@math.kyushu-u.ac.jp}
\date{}
\DeclareMathOperator{\qper}{qper}
\newcommand{\sA}{\underline{A}}
\newcommand{\snu}{\underline{\nu}}

\begin{document}
\begin{abstract}
We study graded and ungraded singularity categories of some commutative Gorenstein toric singularities, namely, Veronese subrings of polynomial rings, and Segre products of some copies of polynomial rings. We show that the graded singularity category has a tilting object whose endomorphism ring is higher representation infinite. Moreover, we construct the tilting object so that the endomorphism ring has a strict root pair of its higher Auslander-Reiten translation, which allows us to give equivalences between singularity categories and (folded) cluster categories in a such a way that their cluster tilting objects correspond to each other.
Our distinguished form of tilting objects also allows us to construct (twisted) Calabi-Yau algebras as the Calabi-Yau completions of the root pairs. We give an explicit description of these twisted Calabi-Yau algebras as well as the higher representation infinite algebras in terms of quivers and relations.
Along the way, we prove that certain idempotent quotients of higher representation infinite algebras remain higher representation infinite.
\end{abstract}

\maketitle
\setcounter{tocdepth}{1}
\tableofcontents
\section{Introduction}
\subsection{Background}
Tilting theory provides an effective and suggestive method in representation theory of commutative rings and finite dimensional algebras \cite{Iy18}. It aims at giving equivalences between various triangulated categories arising from different contexts. The triangulated categories we are interested in in this article are the {\it {\rm(}graded{\rm)} singularity category} \cite{Bu,Or04} of a commutative ring $R$, and the {\it derived category} or the {\it cluster category} \cite{BMRRT,Am09, Ke11} of a finite dimensional algebra $A$. 

%


Let us briefly recall how these categories are related.
Let $R$ be a commutative Gorenstein ring, and let
\[ \sg R:=\Db(\mod R)/\per R \]
be the {\it singularity category} of $R$. Thus, by definition, it is the Verdier quotient of the bounded derived category of finitely generated modules by the perfect derived category. It was introduced by Buchweitz \cite{Bu} and has been actively studied in representation theory \cite{Yo90,LW,Ric2}, algebraic geometry, mirror symmetry \cite{Or04}, and so on.
Also, for a finite dimensional algebra $A$ of finite global dimension over a field $k$ and $n\in\Z$, let
\[ \rC_n(A):=(\per A/\nu_n)_\triangle \]
be its {\it $n$-cluster category} \cite{BMRRT}. Thus it is defined as the triangulated hull \cite{Ke05} of the orbit category of $\per A$ by the shifted Serre functor $\nu_n:=-\lotimes_ADA[-n]$. Such categories have played a crucial role in the categorification of Fomin-Zelevinsky's cluster algebras \cite{CA1}.

The categories $\sg R$ and $\rC_n(A)$ share many important properties; they are (under some assumptions) {Calabi-Yau triangulated categories} endowed with {cluster tilting objects} \cite{Au78,Iy07a,Am09,Guo,Ke11}. In fact, there are many examples where the singularity category $\sg R$ is indeed triangle equivalent to the cluster category $\rC_n(A)$.
More generally, we can consider {\it folded cluster categories} $\rC_n^{(1/a)}(A)=(\per A/-\lotimes_AV)_\triangle$ \cite{Ha3,HaI,Ha7} for some positive integer $a$ and an $a$-th root $V$ of $DA[-n]$ (that is, a complex of $(A,A)$-bimodules such that $V^{\lotimes_Aa}\simeq DA[-n]$). It is a natural quotient of the (ordinary) cluster category $\rC_n(A)$ by a $\Z/a\Z$-action, providing a model to realize the singularity category $\sg R$ as a ``canonical form'' $\rC_n^{(1/a)}(A)$ of Calabi-Yau triangulated categories.


While it is in genral quite difficult to construct equivalences with cluster categories, we may apply the method developed in \cite{HaI} to obtain such an equivalence. Starting from a {\it graded} ring $R=\bigoplus_{i\geq0}R_i$ which is a commutative Gorenstein isolated singularity of dimension $d\geq0$ with Gorenstein parameter $a$, the main result of \cite{HaI} states that we get a commutative diagram of (vertical) equivalences
\begin{equation}\label{diagram}
\xymatrix@R=4mm{
	\sg^\Z\!R\ar[d]^-\rsimeq\ar[r]&\sg^{\Z/a\Z}\!R\ar[r]\ar[d]^\rsimeq&\sg R\ar[d]^-\rsimeq\\
	\per A\ar[r]&\rC_{d-1}(A)\ar[r]&\rC_{d-1}^{(1/a)}(A) }
\end{equation}
as soon as we find a leftmost triangle equivalence. This can be done by finding a single tilting object in $\sg^\Z\!R$ and thus much easier than constructing equivalences directly with cluster categories.
Here, the lower right corner is the $a$-folded cluster category of $A$ obtained from the $a$-th root of $\nu_{d-1}^{-1}$ \cite{Ha7} on $\per A$ corresponding to the degree shift functor on $\sg^\Z\!R$.

\medskip

Note, however, that the above equivalences do not necessarily respect cluster tilting objects even when there are such objects on the singularity category sides.
First of all, the existence of a cluster tilting obejct in the (ordinary) cluster category $\rC_{d-1}(A)$ is guaranteed only under some finiteness conditions on $A$ (see Section \ref{cluster}). Similarly, for the folded cluster category $\rC_{d-1}^{(1/a)}(A)$ to contain a cluster tilting object, the root of $\nu_{d-1}$ we use to construct the category should consist a {\it root pair} in the sense of \cite{Ha7} (see \ref{strict}). 
These constraints require a tilting object we should find in $\sg^\Z\!R$ to be ``as good as possible'' from the point of view of the compatibility with cluster tilting objects we have discussed.

\subsection{Results}
The aim of this paper is to present a class of examples, given by toric singularities, for which one can find ``best possible'' tilting objects. It provides us with a nice description of the folded cluster category, and also implies in particular that the equivalences \eqref{diagram} are compatible with canonical cluster tilting objects. Precisely, we give some commutative Gorenstein rings $R$ which have with cluster tilting objects, and tilting objects $T\in\sg^\Z\!R$ with $\End_{\sg R}^\Z(T)=:A$ satisfying the following, each of which to be explained below.
\begin{enumerate}
\item\label{sroot} The equivalence $\sg^\Z\!R\simeq\per A$ yields a {\it root pair} \cite{Ha7} on $\per A$ which is furthermore {\it strict}.
\item\label{ct} The equivalences in \eqref{diagram} restrict to their canonical cluster tilting subcategories.
\item\label{high} The algebra $A$ is {\it higher representation infinite} \cite{HIO}.
\end{enumerate}

Concerning the points (\ref{sroot}) and (\ref{ct}), recall that an {\it $a$-th root pair} \cite{Ha7} on a finite dimensional algebra $A$ is a pair $(U,P)$ consisting of an $a$-th root $U$ of the inverse of a shifted Serre functor and $P\in\proj A$ satisfying some conditions (see \ref{strict}).
It is easy to see that given an equivalence $\sg^\Z\!R\simeq\per A$ for a commutative graded Gorenstein isolated singularity $R$ of dimension $d$, then the degree shift functor $(1)$ on $\sg^\Z\!R$ yields a root $U$ of $\nu_{d-1}^{-1}$.
However, the exsistence (if any) of a projective module $P$ such that $(U,P)$ is a root pair is a subtle problem. Our construction of a tilting object $T$ do yield such a $P$, which ensures the existence of a cluster tilting object in the folded cluster category.
Another virtue of our construction of $T$ is that the root pair $(U,P)$ is furthermore {\it strict} (\ref{strict}). This gives an understanding of the action of the root $U$ in terms of the finite dimensional algebra $A$ as a certain generalization of reflection functors.

Concerning the point (\ref{high}), this allows us to construct (twisted) Calabi-Yau algebras (or Artin-Schelter regular algebras) \cite{AS,Gi,Ke11} as the Calabi-Yau completion \cite{Ke11,Ha7} of root pairs. Such non-commutative regular algerbas play an important role in non-commutative algebraic geometry especially through the non-commutative projective schemes \cite{AZ} and are closely related to representation theory of finite dimensional algebras via derived categories \cite{MM}. In this way, our results below allow us to relate the derived categories of non-commutative projective schemes and graded singularity categories.

\medskip
The first result is the following, which deals with Veronese subrings of polynomial rings, in other words, certain cyclic quotient singularities.
\begin{Thm}[=\ref{VER}]\label{iVER}
Let $k$ be an arbitrary field, $a\geq1$ and $n\geq2$ be arbitrary integers, $d=an$, and $S=k[x_1,\ldots,x_d]$. Let $R=S^{(n)}$ be the $n$-th Veronese subalgebra. Give a grading on the $R$-module $S$ (in particular on the ring $R$) by requiring a monimial of degree $cn+j$ to have degree $c$ when $0\leq j\leq n-1$. Put
\[ 	T=S\oplus \Om S(1)\oplus\cdots\oplus \Om^{a-1}S(a-1). \]
Then we have the following.
\begin{enumerate}
	\item $T\in\sg^\Z\! R$ is a tilting object.
	\item\label{pair} $A=\End_{\sg R}^\Z(T)$ is $(d-a-1)$-representation infinite and has a strict $a$-th root pair $(U,P)$ of $\nu_{d-a-1}^{-1}$ with $U=\Hom_{\sg R}^\Z(T,\Om T(1))$ and $P=\Hom_{\sg R}^\Z(T,S)$.
	\item We have a commutative diagram of equivalences
	\[ \xymatrix@R=5mm{
	\sg^\Z\!R\ar[r]\ar[d]^-\rsimeq&\sg^{\Z/a\Z}\!R\ar[r]\ar[d]^-\rsimeq&\sg R\ar[d]^-\rsimeq\\
	\Db(\mod A)\ar[r]&\rC_{d-1}(A)\ar[r]&\rC_{d-1}^{(1/a)}(A). } \]
	\item The equivalences in {\rm (3)} restrict to canonical cluster tilting subcategories
	\[ \xymatrix@R=5mm{
		\add\{S(i)\mid i\in\Z\}\ar[r]\ar[d]^-\rsimeq&\add\{S(i)\mid i\in\Z/a\Z\}\ar[r]\ar[d]^-\rsimeq&\add S\ar[d]^-\rsimeq\\
		\add\{ P\lotimes_AU^i[i]\mid i\in\Z\}\ar[r]&\add A\ar[r]&\add P. } \]
\end{enumerate}
\end{Thm}
The cluster tilting object $S\in\sg R$ is given in \cite{Iy07a}, and in \cite{IT} for the graded setting.
Note that different tilting objects in $\sg^\Z\!R$ are constructed in \cite{IT}, but these are not compatible with the orbit construction of cluster categories, in other words, we do not have an analogue of (4) for these tilting objects. In this way, our result can be seen as a partial refinement of \cite{IT}.

For the case $n=d$, the algebra $A=\End_{\sg R}^\Z(T)$ which is isomorphic the stable endomorphism ring $\sEnd_{R}(S)$, is an idempotent quotient of the Beilinson algebra. We refer to \cite{Ka23} for the appearance of this algebra in tilting theory for singular varieties.

As mentioned above, the endomorphism ring $A$ of $T$ being higher representation infinite allows us to construct twisted Calabi-Yau algebras $\Ga$, which is the Calabi-Yau completion of the root pair in \ref{iVER}(\ref{pair}) above. In this case, it is given as the graded ring
\[ \Ga:=\bigoplus_{i\geq0}\sHom_R(S,\Om^iS). \]
We construct equivalences between $\sg^\Z\!R$ and the derived category of the non-commutative projective scheme of $\Ga$, which we denote by $\qper^\Z\!\Ga$.
\begin{Thm}[=\ref{qgr}]
Let $S$, $R$, and $\Ga$ be as above.
\begin{enumerate}
\item The algebra $\Ga$ is a twisted $(d-a)$-Calabi-Yau algebra of Gorenstein parameter $a$.
\item There exist triangle equivalences
\[ \sg^\Z\!R\simeq\Db(\mod A)\simeq\qper^\Z\!\Ga. \]
\item The above equivalences restrict to canonical $(d-1)$-cluster tilting subcategories
\[ \add\{S(i)\mid i\in\Z\}\simeq\add\{ P\lotimes_AU^i[i]\mid i\in\Z\}\simeq\add\{\Ga(i)[i]\mid i\in\Z\}. \]
\end{enumerate}
\end{Thm}
The cluster tilting subcategories for $\sg^\Z\!\simeq\Db(\mod A)$ are as in \ref{iVER}, and those for $\qper^\Z\!\Ga$ are given in \cite{Ha3}.
Furthermore, we give an explicit description of $\Ga$ in terms of quivers and relations, see \ref{Gamma}.

\medskip

Now let us turn to our second result which deals with Segre products of some polynomial rings, in other words, certain torus invariants. Again, we construct a tilting object whose endomorphism ring is higher representation infinite, and an equivalence which respects cluster tilting objects.
\begin{Thm}[=\ref{P^1}, \ref{n=3}]\label{in=3}
Let $R=k[x_1,y_1]\seg \cdots\seg k[x_n,y_n]$ be the Segre product of $n$ copies of polynomial rings with the standard grading.
\begin{enumerate}
	\item There exists a tilting object $T\in\sg^\Z\!R$ such that $\End_{\sg R}^\Z(T)$ is $(n-2)$-representation infinite, which has a square root $\nu_{n-2}^{1/2}$ of $\nu_{n-2}$.
\suspend{enumerate}
Assume in what follows that $n=3$.
\resume{enumerate}
	\item\label{T_0} There is a direct summand $T_0$ of $T$ such that $T=T_0\oplus\nu_1^{-1/2}T_0$. Therefore, the hereditary endomorphism ring $\End_{\sg R}^\Z(T)$ has a strict root pair $(U,P)$ of $\nu_1^{-1}=\tau^{-1}$ with $P=\Hom_{\sg R}^\Z(T,T_0)$.
	\item The algebra $\End_{\sg R}^\Z(T)$ is presented by the following alternating $\widetilde{A_5}$ quiver $Q$ with double arrows between the vertices.
		
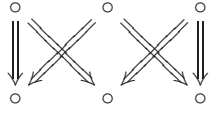
\[ \xymatrix{
		\circ\ar@2[d]\ar@2[dr]&\circ\ar@2[dl]\ar@2[dr]&\circ\ar@2[d]\ar@2[dl]\\
		\circ&\circ&\circ } \]
	\item There exists a commutative diagram of equivalences
	\[ \xymatrix@R=5mm{
		\sg^\Z\!R\ar[r]\ar[d]^-\rsimeq&\sg^{\Z/2\Z}\!R\ar[r]\ar[d]^-\rsimeq&\sg R\ar[d]^-\rsimeq\\
		\Db(\mod kQ)\ar[r]&\rC_3(kQ)\ar[r]&\rC_{3}^{(1/2)}(kQ). } \]
	\item The equivalences in {\rm (4)} restrict to $3$-cluster tilting subcategories
	\[ \xymatrix@R=5mm{
		\add\{T_0(i)\mid i\in\Z\}\ar[r]\ar[d]^-\rsimeq&\add\{T_0(i)\mid i\in\Z/2\Z\}\ar[r]\ar[d]^-\rsimeq&\add T_0\ar[d]^-\rsimeq\\
		\add\{ P\lotimes_AU^i[i]\mid i\in\Z\}\ar[r]&\add kQ\ar[r]&\add P. } \]
\end{enumerate}
\end{Thm}
We refer to \ref{n=3} for the specific form of $T_0$ in (\ref{T_0}).
We note that cluster tilting objects for these rings are given in \cite{HN,Ha6}, and our equivalences with derived or cluster categories give an interpretation in terms of finite dimensional algebras.

\medskip

Finally, we mention that the higher representation infiniteness of the algebra in \ref{in=3}(1) is obtained from a general result, which is of independent interest. In fact, we show that certain special idempotent quotients of higher representation infinite algebras stay higher representation infinite, of smaller dimension. 
\begin{Thm}[=\ref{-1RI}]\label{iRI}
Let $A$ be a $d$-representation infinite algebra with $d\geq1$. Let $e\in A$ be an idempotent satisfying the following.
\begin{itemize}
\item $eA(1-e)=0$.
\item $\id_A\sA<d$.
\item $eAe$ is a semisimple algebra.
\end{itemize}
Then the algebra $A/(e)$ is $(d-1)$-representation infinite.
\end{Thm}
A typical example is the Beilinson algebras. We deduce from this result that ``truncated'' Beilinson algebras are higher representation infinite (\ref{trunc}).

\subsection*{Acknowledgements}
The author would like to thank Martin Kalck and Osamu Iyama for valuable discussions.

\section{Reduction of higher representation infinite algebras}\label{red}
Let us start with the following notion which is one of main objects of this paper. For each $n\in\Z$ we denote by $\nu_n=-\lotimes_ADA[-n]$ the shifted Serre functor, and its inverse $\nu_n^{-1}=\RHom_A(DA,-)[n]$ on $\Db(\mod A)$.
\begin{Def}[\cite{HIO}]
Let $n\geq0$. A finite dimensional algebra $A$ is called {\it $n$-representation infinite} if $\gldim A\leq n$ and
\[ \nu_n^{-i}A\in\mod A \subset \Db(\mod A) \]
for all $i\geq0$.
\end{Def}
Let $A$ be a $d$-representation infinite algebra, and let $1\neq e\in A$ be an idempotent. The aim of this section is to prove that certain quotient $\sA=A/(e)$ is still higher representation infinite.

Our (rather strong) assumption on the idempotent is the following.
\begin{Assum}\label{ass}
\begin{enumerate}
	\item\label{str} $eA(1-e)=0$.
	\item $\id_A\sA<d$.
	\item\label{wtw} $eAe$ is a semisimple algebra.
\end{enumerate}
\end{Assum}

This ensures that the quotient algebra $A/(e)$ remains to be higher representation infinite.
\begin{Thm}\label{-1RI}
Let $A$ be a $d$-representation infinite algebra with $d\geq1$. Under the assumptions in \ref{ass}, the algebra $A/(e)$ is $(d-1)$-representation infinite.
\end{Thm}

We denote by $H^0\circ\nu_n^{-1}=\tau_n^{-1}$ on $\mod A$, the (inverse) $n$-Auslander-Reiten translation.
Let us introduce a subcategory $\mathcal{Q}$ of $\mod A$.
\begin{Def}\label{DefQ}
	Let $A$ be a finite dimensional algebra and $d\geq1$. We set
	\[ \mathcal{Q}=\{ X\in\mod A\mid \Ext_A^i(DA,\tau_d^{-l}X)=0 \text{ for all } i\neq d \text{ and } l\geq0 \}. \]
\end{Def}
We have the following reformulation of the definition of $\cQ$ in terms of derived categories.
\begin{Lem}\label{LemQ}
	Let $A$ be a finite dimensional algebra of global dimension $\leq d$.
	\begin{enumerate}
		\item We have $\mathcal{Q}=\{X\in\mod A\mid \nu_d^{-l}X \in \mod A \text{ for all } l\geq0 \}$. In particular, $\mathcal{Q}$ is closed under $\nu_d^{-1}$.
		\item If $A$ is $d$-representation infinite, then $\mathcal{Q}$ is a resolving subcategory of $\mod A$ containing the preprojective modules.
	\end{enumerate}
\end{Lem}
\begin{proof}
	(1)  For each $Y\in\mod A$, we have $\Ext^i_A(DA,Y)=0$ for $i\neq d$ if and only if $\nu_d^{-1}Y\in\mod A$, and in this case $\Ext^d_A(DA,Y)=\nu_d^{-1}Y=\tau_d^{-1}Y$. Then the assertion easily follows by induction.
	
	(2)  Straightforward.
\end{proof}

Let us collect some easy consequences of \ref{ass} which will be frequently used.
\begin{Lem}\label{easy}
	\begin{enumerate}
		\item The functor $\Db(\mod\sA)\to\Db(\mod A)$ is fully faithful.
		\item $AeA\simeq Ae$ in $\mod A^\op$.
		\item\label{fac} For any $X\in\mod A$, the triangle $X\lotimes_AAeA\to X \to X\lotimes_A\sA$ in $\Db(\mod A)$ is isomorphic to the one coming from the exact sequence $0\to XeA \to X \to X/XeA \to0$ in $\mod A$.
	\end{enumerate}
\end{Lem}
\begin{proof}
	The assertions (1)(2) follow from \ref{ass}(\ref{str}), so we prove (\ref{fac}). By (2) we have $X\lotimes_AAeA=X\otimes_AAe=Xe$ which is same as $XeA$. Therefore the morphism $X\lotimes_AAeA\to X$ in $\Db(\mod A)$ is isomorphic to the inclusion $XeA\to X$ in $\mod A$, hence we get the assertion.
\end{proof}

\begin{Lem}
	For any $X\in\mod A$, we have $X\lotimes_AAeA\in\add eA\subset\mathcal{Q}$.
\end{Lem}
\begin{proof}
	We have seen in \ref{easy}(\ref{fac}) that $X\lotimes_AAeA=Xe\lotimes_{eAe}eA=Xe\otimes_{eAe}eA\in\mod A$. By \ref{ass}(\ref{wtw}), we deduce that this is in $\add eA$.
\end{proof}

We denote by $\nu$ (resp. $\snu$) the Serre functor on $\Db(\mod A)$ (resp. on $\Db(\mod\sA)$). We have the following behavior of these Serre functors under the inclusion $\Db(\mod\sA)\hookrightarrow\Db(\mod A)$.
\begin{Lem}\label{bar}
	For each $X\in\Db(\mod\sA)$, we have $\snu^{-1}X=(\nu^{-1}X)\lotimes_A\sA$.
\end{Lem}
\begin{proof}
	This is nothing but \cite[3.10(c)]{Ke11}. For the convenience of the reader, we include the proof. Note first that $\RHom_{\sA}(D\sA,\sA)=\sA\lotimes_A\RHom_A(DA,A)\lotimes_A\sA$. Indeed, this is a consequence of the formula $Y\lotimes_A\RHom_A(M,N)\lotimes_AX\ysimeq\RHom_A(\RHom_{A^\op}(X,M),Y\lotimes_AN)$. 
	Then it follows that we have $\snu^{-1}X=X\lotimes_{\sA}\RHom_{\sA}(D\sA,\sA)=X\lotimes_A\RHom_A(DA,A)\lotimes_A\sA=(\nu^{-1}X)\lotimes_A\sA$.
\end{proof}

Now we are ready to prove our result.
\begin{proof}[Proof of \ref{-1RI}]
	We show by induction on $i$ that $\nu_{d-1}^{-1}(\snu_{d-1}^{-i+1}\sA)\in\mathcal{Q} \subset \md A$ for all $i\geq0$, where we formally understand this as $A\in\mathcal{Q}$ for $i=0$.
	
	Let us first note that this claim will yield our desired result. Indeed, we know by \ref{bar} that $\snu_{d-1}^{-i}\sA=\nu_{d-1}^{-1}(\snu_{d-1}^{-i+1}\sA)\lotimes_A\sA$, and by \ref{easy}(\ref{fac}) that applying $-\lotimes_A\sA$ to an $A$-module is nothing but factoring out the trace of $eA$. Therefore the quotient module lies in $\md\sA$.
	
	Now we prove our claim. It is trivial for $i=0$. Suppose we have the assertion for $i$, so there exists an exact sequence
	\[ \xymatrix{ 0\ar[r]& P_i\ar[r]&\nu_{d-1}^{-1}(\snu_{d-1}^{-i+1}\sA)\ar[r]& \snu_{d-1}^{-i}\sA \ar[r]& 0} \]
	in $\md A$ with $P_i\in\add eA$, where we replace the middle term by $A$ for $i=0$. Now we apply $\nu_d^{-1}=\RHom_A(DA,-)[d]$ to this exact sequence. Since the first two terms are in $\mathcal{Q}$, the resulting complexes concentrate in degree $0$. On the other hand, since the injective dimension of $\sA$ over $A$ is $<d$, so is the $\sA$-module $\snu_{d-1}^{-i}\sA$. Then the complex $\nu_d^{-1}(\snu_{d-1}^{-i}\sA)$ concentrates in degree $\leq-1$, hence in degree $-1$. Therefore we obtain an exact sequence
	\[ \xymatrix{ 0\ar[r]& \nu_{d-1}^{-1}(\snu_{d-1}^{-i}\sA)\ar[r]& \nu_d^{-1}P_i\ar[r]&\nu_d^{-1}\nu_{d-1}^{-1}(\snu_{d-1}^{-i+1}\sA)\ar[r]&0 }. \]
	Now the middle term is preprojective and the last term is in $\mathcal{Q}$ by induction hypothesis and \ref{LemQ}, we deduce that the first term is in $\mathcal{Q}$, completing the induction step.
\end{proof}

\begin{Ex}\label{trunc}
For each $d\geq1$, let $A$ be the $d$-Beilinson algebra which is presented by the following quiver with $(d+1)$-fold arrows and the commutativity relations.
\[ \xymatrix{ 0\ar@3[r]&1\ar@3[r]&\cdots\ar@3[r]&d-1\ar@3[r]&d } \]
It is a fundamental example of a $d$-representation infinite algebra \cite[2.15]{HIO}.

(1)  Let $e$ be the idempotent of $A$ corresponding to the vertex $0$. It is easy to see that it satisfies the assumptions in \ref{ass}. We see by \ref{-1RI} that the truncated Beilinson algebra
\[ \xymatrix{ 1\ar@3[r]&\cdots\ar@3[r]&d-1\ar@3[r]&d } \]
is $(d-1)$-representation infinite.

(2)  We can iterate this process to deduce that for each $0\leq n\leq d$, the algebra
\[ \xymatrix{ 0\ar@3[r]&1\ar@3[r]&\cdots\ar@3[r]&n-1\ar@3[r]&n } \]
with $(d+1)$-fold arrows and the commutativity relations is $n$-representation infinite. (Note that the vertices are renumbered.)
\end{Ex}

The following example shows that we cannot drop a somewhat strong assumption \ref{ass}(\ref{wtw})
\begin{Ex}
	Let $\Ga$ be the skew group algebra $k[x,y,z,w]\ast G$ for the cyclic subgroup $G$ generated by $\diag(\zeta,\zeta,\zeta,\zeta^2)\in\SL_4(k)$ where $\zeta$ is a primitive $5$-th root of unity and $\ch k\neq5$.
	Then $\Ga$ is presented by the McKay quiver of $G$ with commutativity relations as in the leftmost figure below.
	We give a grading on $\Ga$ by declaring the arrows $i\to j$ with $i>j$ to have degree $1$, and degree $0$ otherwise. Then $\Ga$ is a $4$-Calabi-Yau algebra of Gorenstein parameter $1$ by \cite[5.6]{AIR}, and hence its degree $0$ part $A$ which is presented by the following quiver in the middle with commutativity relations, is $3$-representation infinite.
	
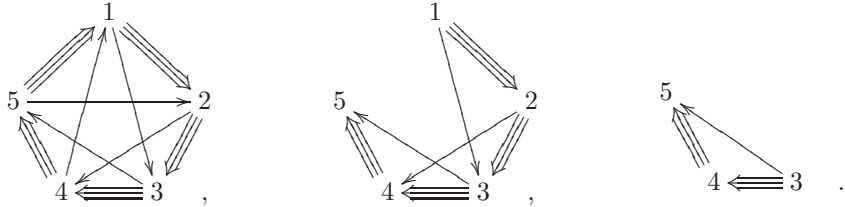
\[
	\xymatrix@C=2.5mm@R=7.5mm{
		&& 1\ar@3[drr]\ar[ddr]&&\\
		5\ar@3[urr]\ar[rrrr]&&&&2\ar@3[dl]\ar[dlll]\\
		&4\ar@3[ul]\ar[uur]&&3\ar@3[ll]\ar[ulll]&, }\qquad\qquad
	\xymatrix@C=2.5mm@R=7.5mm{
		&& 1\ar@3[drr]\ar[ddr]&&\\
		5&&&&2\ar@3[dl]\ar[dlll]\\
		&4\ar@3[ul]&&3\ar@3[ll]\ar[ulll]&, }\qquad\qquad
	\xymatrix@C=2.5mm@R=7.5mm{
		&&&&\\
		5&&&&\\
		&4\ar@3[ul]&&3\ar@3[ll]\ar[ulll]&. } \]
	It is easily verified that the injective dimensions of the simple $A$-modules are $3, 3, 2, 1, 0$. Take $e=e_1+e_2$. Then $\sA=A/(e)$, which is presented by the quiver on the right, is not $2$-representation infinite. Indeed, $\Ext^1_{\sA}(D{\sA},{\sA})\supset\Ext^1_{\sA}(I_5,P_3)=\Ext^1_{\sA}(S_5,S_3)\neq0$.
\end{Ex}

\begin{Ex}
	We note that the converse does not hold, that is, let $A$ be a finite dimensional algebra of global dimension $d$, and let $e\in A$ be an idempotent as in \ref{ass} such that $A/(e)$ is $(d-1)$-representation infinite. Then $A$ is not necessarily $d$-representation infinite.
	
	We have the following example in $d=2$. Let $A$ be the algebra defined by the following quiver with commutativity relations and $e=e_1$.
	\[ \xymatrix{ 1\ar@2[r]^x_y&2\ar@2[r]^x_y&3 } \]
	Then $A/(e_1)$ is $1$-representation infinite, but $A$ is not $2$-representation infinite. Indeed, we have $\id P_2=1$.
\end{Ex}


We end this section by recording the dual of \ref{-1RI}.
\begin{Thm}\label{-1RI*}
Let $A$ be a $d$-representation infinite algebra. Let $e\in A$ be the idempotent satisfying the following.
\begin{itemize}
\item $(1-e)Ae=0$.
\item $\pd_AD\sA<d$.
\item $eAe$ is a semisimple algebra.
\end{itemize}
Then $A/(e)$ is $(d-1)$-representation infinite.
\end{Thm}

\section{Folded cluster categories and strict root pairs}\label{cluster}
Let us recall the notion of folded cluster categories which are finite cyclic quotient of (oridnary) cluster categories, which appear naturally as a model of singularity categories.

For our purposes, we restrict ourselves to the folded cluster categories of finite dimensional algebras. Recall that a finite dimensional algebra $A$ is {\it $\tau_d$-finite} \cite{Iy11} if $\gldim A\leq d$ and $\Ext^d_A(DA,A)^{\otimes_An}=0$ for sufficiently large $n$. For a positive integer $a$, an {\it $a$-th root} of a bimodule complex $X$ is a bimodule complex $U$ satisfying $U^{\lotimes_Aa}\simeq X$ in $\rD(A^e)$.
\begin{Def}[\cite{HaI,Ha7}]
Let $A$ be a finite dimensional algebra which is $\tau_d$-finite, and let $U$ be an $a$-th root of $\RHom_A(DA,A)[d]$. We define the {\it $a$-folded cluster category $\rC_d^{(1/a)}(A)$} as the triangulated hull of the orbit category $\Db(\mod A)/-\lotimes_AU$, thus we have a diagram
\[ \xymatrix{ \Db(\mod A)\ar[r]&\Db(\mod A)/-\lotimes_AU\ar@{^(->}[r]&\rC_d^{(1/a)}(A) }. \]
\end{Def}

Note that in this generality, the folded cluster category is just an orbit category, and further properties like the existence of a cluster tilting object is not guaranteed. 
We need some additional data called a root pair \cite{Ha7} in the following sense.
For a bimodule complex $X$ over an algebra $A$ and $i\geq0$, we denote by $X^i$ the $i$-fold derived tensor power of $X$ over $A$.
\begin{Def}\label{strict}
	Let $A$ be a finite dimensional Iwanaga-Gorenstein algebra and $n\in\Z$. A {\it strict $a$-th root pair} on $A$ is a pair $(U,P)$ consisting of $U\in\rD(A^e)$ and $P\in\proj A$ satisfying the following.
	\begin{itemize}
		\item $U^a\simeq\RHom_A(DA,A)[n]$ in $\rD(A^e)$, where the tensor factor is $a$ times.
		\item $\RHom_A(P\lotimes_AU^{j},P\lotimes_AU^{i})=0$ for every $0\leq i<j\leq a-1$.
		\item $\add A=\add(\bigoplus_{i=0}^{a-1}P\lotimes_AU^{i})$.
	\end{itemize}
\end{Def}
The first condition says that the functor $-\lotimes_AU$ is an $a$-th root of (the inverse of) the shifted Serre functor. The second and the third conditions together say that $\per A$ has a semi-orthogonal decomposition $\per A=\thick (P\lotimes_AU^{a-1})\perp\cdots\perp\thick P$ with each of its component the thick closure of a projective module.

We have the following result on cluster tilting subcategories and folded cluster categories arising from strict root pairs, which is a folded analogue of \cite{Iy11,Am09}.
\begin{Thm}[\cite{Ha7}]
	Let $A$ be a finite dimensional algebra which is $\tau_d$-finite, and let $(U,P)$ be a strict $a$-th root pair for $\RHom_A(A,A^e)[d]$.
	\begin{enumerate}
		\item The subcategory $\rU:=\add\{ P\lotimes_A U^i\mid i\in\Z\}\subset\Db(\mod A)$ is a $d$-cluster tilting subcategory.
		\item The object $P\in\rC_d^{(1/a)}(A)$ is a $d$-cluster tilting object.
	\end{enumerate}
	Therefore, the orbit functor restricts to cluster tilting subcategories, giving rise to the commutative diagram
	\[ \xymatrix@R=3mm{
		\Db(\mod A)\ar[r]\ar@{}[d]|-\vsubset&\rC_d^{(1/a)}(A)\ar@{}[d]|-\vsubset\\
		\rU\ar[r]&\add P. } \]
\end{Thm}
We shall see that for certain (graded and ungraded) singularity categories of commutative rings with a canonical cluster tilting object, we get the isomorphic diagram as above, realizing the singularity category as its canonical form together with the information on cluster tilting objects.

\medskip
In our examples, we are given an algebraic triangulated category $\T$ with an $a$-th root $F$ of a shifted Serre functor, and we have a tilting object $T$, so that there is an equivalence $\T\simeq\per A$ for $A=\End_\T(T)$. It is natural to ask whether $F$ forms a root pair on $A$ for some projective module. 
The following is a straightforward reformulation of \ref{strict} in terms of tilting objects. We say that an $a$-th root $F$ of $\nu_n$ in an algebraic triangulated category is {\it algebraic} if $F$ is given by a bimodule and $F^a\simeq\nu_n$ lifts to an isomorphism of bimodules (over a fixed enhancement).
\begin{Lem}\label{tilt}
	Let $\T$ be an algebraic triangulated category with a Serre functor $\nu$, let $n\in\Z$ and $a\in\Z_{>0}$. Suppose that there is an algebraic $a$-th root of $\nu_n$.
	If there is $T_0\in\T$ satisfying
	\begin{enumerate}
		\renewcommand\labelenumi{(\roman{enumi})}
		\item $T:=\bigoplus_{i=0}^{a-1}F^iT_0$ is a tilting object of $\T$, and 
		\item $\Hom_\T(F^jT_0,F^iT_0[l])=0$ for all $0\leq i<j\leq a-1$ and $l\in\Z$,
	\end{enumerate}
	then $A=\End_\T(T)$ has a strict root pair $(U,P)$ with $P=\Hom_\T(T,T_0)$ and $U=\RHom_\A(T,FT)$ for an enhancement $\A$ of $\T$.
\end{Lem}
We will therefore look for, and do construct, a tilting object of the form (i) above.

\section{Veronese subrings}
\subsection{Tilting theory}
The main result of this section is the existence of a tilting object whose endomorphism algebra is higher representation infinite, moreover with a strict root pair on it. 
\begin{Thm}\label{VER}
	Let $k$ be an arbitrary field, $a\geq1$ and $n\geq2$ be arbitrary integers, $d=an$, and $S=k[x_1,\ldots,x_d]$. Let $R=S^{(n)}$ be the $n$-th Veronese subalgebra. Give a grading on the $R$-module $S$ (in particular on the ring $R$) by requiring a monimial of degree $cn+j$ to have degree $c$ when $0\leq j\leq n-1$. Put
	\[ 	T=S\oplus \Om S(1)\oplus\cdots\oplus \Om^{a-1}S(a-1). \]
	Then we have the following.
	\begin{enumerate}
		\item $T\in\sg^\Z\! R$ is a tilting object.
		\item $A=\sEnd_R^\Z(T)$ is $(d-a-1)$-representation infinite with a strict $a$-th root pair for $\nu_{d-a-1}$.
		\item We have a commutative diagram of equivalences
		\[ \xymatrix@R=5mm{
			\sg^\Z\!R\ar[r]\ar[d]^-\rsimeq&\sg^{\Z/a\Z}\!R\ar[r]\ar[d]^-\rsimeq&\sg R\ar[d]^-\rsimeq\\
			\Db(\mod A)\ar[r]&\rC_{d-1}(A)\ar[r]&\rC_{d-1}^{(1/a)}(A). } \]
	\end{enumerate}
\end{Thm}
Decompose $S=\bigoplus_{i=0}^{n-1}M_i$ as graded $R$-modules, where $M_i=\bigoplus_{a\geq0}S_{cn+i}$ is the graded $R$-module whose degree $c$ part consists of polynomials in $x_1,\ldots,x_d$ of degree $cn+i$.
\begin{Rem}
	\begin{enumerate}
		\item $a$ is the Gorenstein parameter of $R$.
		\item When $k$ has characteristic prime to $n$ and has all $n$-th root of unity, our grading on $R$ and $M_i$'s can be reformulated using the McKay quiver as follows. Let $G$ be the cyclic subgroup of $\SL_d(k)$ generated by $\diag(\zeta,\ldots,\zeta)$ for a fixed primitive $n$-th root of unity $\zeta$. Then $G$ naturally acts on $S$, and the invariant subring is nothing but our Veronese subring $R$. Also we know that the skew group algebra $S\ast G$ is isomorphic to $\End_R(S)$ \cite{Au62}\cite[10.8]{Yo90}, and is presented by the following McKay quiver of $G$ (with $d$-fold arrows between the vertices) and commutativity relations \cite{BSW}.
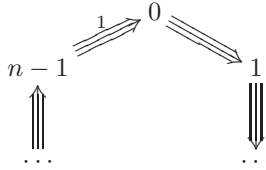
		\[ \xymatrix@R=3mm{
			&0\ar@3[dr]&\\
			n-1\ar@3[ur]^-1&&1\ar@3[dd] &\\
			\\
			\cdots\ar@3[uu]&&\cdots} \]
		Now give a grading on this McKay quiver by setting the arrows from $n-1$ to $0$ to have degree $1$, and the remaining arrows to have degree $0$.
		If $e_i$ is the idempotent of $S\ast G$ corresponding to vertex $i$, then we have $R=e_0(S\ast G)e_0$ and $M_i=e_i(S\ast G)e_0$ as an $R$-module. Then the grading of $R$ and $M_i$ induced from that of $S\ast G$ is precisely the one we gave above.
		\item The $(d-a-1)$-representation infiniteness of $A$ can be rephrased that the summands of $T$ form a full strong exceptional sequence $(M_1,\ldots,M_{n-1},\Om M_1(1),\ldots,\Om^{a-1}M_{n-1}(a))$ in $\sg^\Z\!R$ which is {\it geometric} in the sense of \cite[Section 2]{BP}.
	\end{enumerate}
\end{Rem}

We define the syzygy $\Om$ as the kernel of the projective cover in $\md^\Z\!R$. Also we let $\Om^{-1}$ be the cosyzygy, the cokernel of the minimal left projective approximation in $\md^\Z\!R$. Let us start with the following simple observation.
\begin{Lem}\label{pure}
	Let $R=\bigoplus_{i\geq0}R_i$ be a positively graded algebra such that $R_0$ is semisimple. Suppose that
	\[ \xymatrix{ 0\ar[r]&X\ar[r]&Y\ar[r]&Z\ar[r]&0 } \]
	is an exact sequence of modules without projective summands such that each term is generated in degree $c\in\Z$. Then there is an exact sequence of syzygies
	\[ \xymatrix{ 0\ar[r]&\Om X\ar[r]&\Om Y\ar[r]&\Om Z\ar[r]&0 }. \]
\end{Lem}
\begin{proof}
	Note that a graded $R$-module $M$ is generated in degree $c$ if and only if the natural map $M_c\otimes_{R_0}R\to M$ is surjective. If $R_0$ is semisimple, then this map is a projective cover. Suppose that $0\to X\to Y\to Z\to 0$ is an exact sequence of graded modules generated in degree $c$. Then by the above remark we have an exact sequence of projective covers, which yield the following commutative diagram of exact sequences.
	\[ \xymatrix{
		0\ar[r]&X_c\otimes_{R_0}R\ar[r]\ar@{->>}[d]& Y_c\otimes_{R_0}R\ar[r]\ar@{->>}[d]&Z_c\otimes_{R_0}R\ar[r]\ar@{->>}[d]&0\\
		0\ar[r]&X\ar[r]&Y\ar[r]&Z\ar[r]&0 } \]
	Taking the kernels we obtain the desired exact sequence of syzygies. 
\end{proof}

Let $V$ be the vector space with basis $\{x_1,\ldots,x_d\}$. Then the Koszul complex of $S=k[x_1,\ldots,x_d]$ is an exact sequence
\[ \xymatrix{ 0\ar[r]&S\otimes\bigwedge^dV\ar[r]&\cdots\ar[r]&S\otimes V\ar[r]&S\ar[r]&k\ar[r]&0 } \]
whose maps are $S$-linear maps such that
\[ S\otimes\Wedge^pV\to S\otimes\Wedge^{p-1}V, \quad 1\otimes(x_{i_1}\wedge\cdots\wedge x_{i_p})\mapsto \sum_{j=1}^p(-1)^jx_{i_j}\otimes(x_{i_1}\wedge\cdots\wedge\widehat{x_{i_j}}\wedge\cdots\wedge x_{i_p}). \]
Decomposing this complex of $S$-modules over $R$, we obtain the $(d-1)$-almost split sequences and $(d-1)$-fundamental sequence in $\add S\subset\CM R$ \cite{Iy07a,Iy07b}.
Taking our grading into account, the almost split sequence in $\add\{S(i)\mid i\in\Z\}\subset\CM^\Z\!R$ starting at $M_i$ is given as follows.
\[ \xymatrix@!C=4mm@!R=2mm{
	&M_{i+1}^\oplus\ar[rr]\ar[dr]&&\cdots\ar[rr]&&M_{n-1}^\oplus\ar[rr]\ar[dr]&&R(1)^\oplus\ar[rr]\ar[dr]&&M_1(1)^\oplus\ar[rr]&&\cdots\ar[rr]\ar[dr]&&M_{i-1}(a)^\oplus\ar[dr]&\\
	M_i\ar[ur]&&N_{i,i+1}\ar[ur]&& &&N_{i,n-1}\ar[ur]&&N_{i,n}(1)\ar[ur]&& && N_{i,i+d-2}(n)\ar[ur]&&M_i(a) } \]
Here we have defined $N_{i,l}$ so that there is a surjection $M_j(c)^{\oplus}\to N_{i,cn+j}(c)$ in $\md^\Z\!R$, where we will always write $l=cn+j$ with $0\leq j\leq n-1$.

Let us observe that these sequences show $T=S\oplus\Om S(a)\oplus\cdots\Om^{a-1}S(a-1)$ generates the graded singularity category.
\begin{Lem}\label{seisei}
	We have $\thick T=\sg^\Z\!R$.
\end{Lem}
\begin{proof}
	Since $\add\{S(i)\mid i\in\Z\}\subset\sg^\Z\!R$ is $(d-1)$-cluster tilting and thus generates $\sg^\Z\!R$, we only have to show $S(i)\in\thick T$ for all $i\in\Z$. The almost split sequence starting at $M_i$ above shows $M_i(a)\in\thick (T\oplus M_1(a)\oplus\cdots \oplus M_{i-1}(a))$ and similarly the one ending at $M_i(a-1)$ shows $M_i(-1)\in\thick(M_{i+1}(-1)\oplus\cdots\oplus M_{n-1}(-1)\oplus T)$. We deduce inductively $S(i)\in\thick T$ for all $i\in\Z$.
\end{proof}

We also need some information on the terms $M_i$ and $N_{i,l}$ as well as their syzygies.
\begin{Lem}\label{ind}
	We have the following for all $p\in\Z$.
	\begin{enumerate}
		\item\label{syz} The $p$-th syzygies of the terms in the almost split sequence and some projective modules form the almost split sequence
		\[ \xymatrix@!C=4mm@!R=2mm{
			&\Om^p M_{i+1}^\oplus\ar[rr]\ar[dr]&&\cdots\ar[rr]&&\Om^p M_{n-1}^\oplus\ar[rr]\ar[dr]&&R(-p+1)^\oplus\ar[rr]\ar[dr]&&\Om^p M_1(1)^\oplus\ar[rr]&&\cdots\ar[rr]\ar[dr]&&\Om^p M_{i-1}(a)^\oplus\ar[dr]&\\
			\Om^p M_i\ar[ur]&&\Om^pN_{i,i+1}\ar[ur]&& &&\Om^p N_{i,n-1}\ar[ur]&&\Om^p N_{i,n}(1)\ar[ur]&& && \Om^p N_{i,i+d-2}(a)\ar[ur]&&\Om^p M_i(a) } \]
		for all $1\leq i\leq n-1$.
		\item\label{gen} $\Om^pS$ is generated in degree $p$.
		\item\label{genN} For each $1\leq i\leq n-1$, and $i+1\leq l\leq i+d-2$, the module $\Om^pN_{i,l}$ is generated in degree $p$.
	\end{enumerate}
\end{Lem}
\begin{proof}
	We give a proof along one induction. Consider the following statements for $p\geq0$.
	\begin{enumerate}
		\renewcommand{\labelenumi}{(\alph{enumi})$_{p}$}
		\renewcommand{\theenumi}{\alph{enumi}}
		\item\label{SS} The conclusion of (\ref{syz}) for $p$.
		\item\label{MM} $\Om^pM_i$ is generated in degree $p$ for all $1\leq i\leq n-1$.
		\item\label{NN} $\Om^pN_{i,l}$ is generated in degree $p$ for all $1\leq i\leq n-1$ and $i+1\leq l\leq i+d-2$.
	\end{enumerate}
	\newcommand{\refm}[2]{(\ref{#1})$_{#2}$}
	We have \refm{SS}{0} as the initial almost split sequence. We prove that \refm{SS}{p} gives \refm{MM}{p} and \refm{NN}{p}, and that \refm{SS}{p}, \refm{MM}{p} and \refm{NN}{p} imply \refm{SS}{p+1}, completing the induction.
	
	Suppose \refm{SS}{p}. By the last map $R(-p+a)^\oplus\to \Om^pM_1(a)$ of the almost split sequence at $M_1$, we see that $\Om^pM_1$ is generated in degree $p$. Similarly, by the last map $\Om^pM_{i-1}(a)^\oplus\to\Om^pM_i(a)$ of the almost split sequence at $M_i$, we see that $\Om^pM_i$ is generated in degree $p$ if $\Om^pM_{i-1}$ is, which inductively gives \refm{MM}{p}. Also there are surjections $R(-p+c)^\oplus\to \Om^p N_{i,cn}(c)$ for $1\leq c\leq a-1$ and $\Om^pM_j(c)^\oplus\to \Om^pN_{i,cn+j}(c)$ for $cn+j\in\{i+1,\ldots,i+d-2\}\setminus\{n,2n,\ldots,(a-1)n\}$, which together with \refm{MM}{p} implies \refm{NN}{p}.
	
	Let us next assume \refm{SS}{p}, \refm{MM}{p} and \refm{NN}{p}. Then the long exact sequence in \refm{SS}{p} is a splicing of short exact sequences of the form
	\begin{itemize}
		\item $0\to\Om^pN_{i,cn+j-1}(c)\to\Om^p M_j(c)^\oplus\to \Om^p N_{i,cn+j}(c)\to 0$ for some $1\leq j\leq n-1$ and $cn+j\in\{i+1,\ldots,i+d-2\}\setminus\{n,2n,\ldots,(a-1)n\}$, or
		\item $0\to \Om^pN_{i,cn-1}(c-1)\to R(-p+c)^\oplus\to \Om^pN_{i,cn}(c)\to0$ for some $1\leq c\leq a-1$.
	\end{itemize}
	By \refm{MM}{p} and \refm{NN}{p}, the terms in the exact sequence of the first type is generated in degree $p-c$. It follows from \ref{pure} that their syzygies form an exact sequence $0\to\Om^{p+1}N_{i,cn+j-1}(c)\to\Om^{p+1} M_j(c)^\oplus\to \Om^{p+1} N_{i,cn+j}(c)\to 0$. Also since $\Om^{p+1}N_{i,cn}(c)=\Om(\Om^{p}N_{i,cn}(c))=\Om^pN_{i,cn-1}(c-1)$ is generated in degree $p-c+1$ by \refm{NN}{p}, there is an exact sequence $0\to \Om^{p+1}N_{i,cn-1}(c-1)\to R(-p-1+c)^\oplus\to \Om^{p+1}N_{i,cn}(c)\to0$. Connecting these exact sequences of syzygies, we obtain the desired sequence in \refm{SS}{p+1}. This finishes the proof for $p\geq0$.
	
	Now we discuss negative syzygies by means of the duality $(-)^\ast=\Hom_R(-,R)$ on $\CM^\Z\!R$.\\
	(\ref{gen})  Note that we have $M_i^\ast\simeq M_{n-i}(-1)$ in $\md^\Z\!R$ for $1\leq i\leq n-1$, thus $S^\ast\simeq S(-1)$. Let $p\geq0$. If $\Om^pS$ is generated in degree $p$, then there is a monomorphism $\Om^{p+1}S\hookrightarrow R(-p)^\oplus$ which is a left projective approximation since its cokernel is $\Om^pS\in\CM^\Z\!R$. Applying $(-)^\ast$ gives a surjection $R(p)^\oplus\twoheadrightarrow \Om^{-p-1}S(-1)$, which shows that $\Om^{-p-1}S$ is generated in degree $-p-1$ for each $p\geq0$, as desired.\\
	(\ref{genN})  Similarly to above we compute $N_{i,l}^\ast$. By the uniquenss of almost split sequences, applying $(-)^\ast$ to the almost split sequence starting at $M_i$ gives the almost split sequence ending at $M_i^\ast=M_{n-i}(-1)$. It follows that we have $N_{i,cn+j}^\ast=N_{n-i,(a-c)n+(n-j-1)}(-1)$. Now let $p\geq0$. Then we have an injective left projective approximation $\Om^{p+1}N_{n-i,(a-c)n+(n-j-1)}\hookrightarrow R(-p)^\oplus$ since $\Om^pN_{n-i,(a-c)n+(n-j-1)}$ is generated in degree $p$. Applying $(-)^\ast$ gives a surjection $R(p)^\oplus\twoheadrightarrow \Om^{-p-1}N_{i,cn+j}(-1)$, which shows $\Om^{-p-1}N_{i,cn+j}$ is generated in degree $-p-1$ for all $p\geq0$.\\
	(\ref{syz})  This is shown using $(-)^\ast$. Since this assertion is not used later in this paper, we leave the details to the reader.
\end{proof}

Now we give the following computation which is crucial.
\begin{Lem}\label{conc}
	\begin{enumerate}
		\item\label{pS} $\sHom_R(S,\Om^pS)$ is concentrated in degree $p$ for each $p\geq0$.
		\item\label{pN} $\sHom_R(S,\Om^pN_{i,l})$ is concentrated in degree $p$ for each $p\geq0$ and $i,l$.
	\end{enumerate}
\end{Lem}
\begin{proof}
	We proceed by induction on $p$. Consider the following statements for $p\geq0$.
	\begin{enumerate}
		\renewcommand{\labelenumi}{(\alph{enumi})$_{p}$}
		\renewcommand{\theenumi}{\alph{enumi}}
		\item\label{MMM} $\sHom_R(S,\Om^pM_i)$ is concentrated in degree $p$ for all $1\leq i\leq n-1$.
		\item\label{NNN} $\sHom_R(S,\Om^pN_{i,l})$ is concentrated in degree $p$ for all $1\leq i\leq n-1$ and $i+1\leq l\leq i+d-2$.
	\end{enumerate}
	\newcommand{\refm}[2]{(\ref{#1})$_{#2}$}
	We start with \refm{NNN}{-1}, which we understand as a true statement by $\sHom_R(S,\Om^{-1}N_{i,l})=\Ext^1_R(S,N_{i,l})=0$. We prove \refm{NNN}{p-1}+\refm{MMM}{p} $\Rightarrow$ \refm{NNN}{p} and \refm{NNN}{p} $\Rightarrow$ \refm{MMM}{p+1}, which completes the induction.
	
	Suppose \refm{NNN}{p-1} and \refm{MMM}{p}. Note that all the terms $\Om^pN_{i,l}$ appear either as
	\begin{itemize}
		\item $0\to \Om^pN_{i,cn+j-1}\to \Om^pM_j^{\oplus}\to \Om^pN_{i,cn+j}\to0$ where $1\leq i,j\leq n-1$, or
		\item $0\to \Om^pN_{i,cn-1}(-1)\to R(-p)^\oplus\to\Om^pN_{i,cn}\to 0$
	\end{itemize}
	in the exact sequence in \ref{ind}(\ref{syz}).
	Applying $\sHom_R(S,-)$ to the exact sequence of the second kind yields an isomorphism of graded $R$-modules $\sHom_R(S,\Om^pN_{i,cn})\xrightarrow{\simeq}\sHom_R(S,\Om^{p-1}N_{i,cn-1})(-1)$. Since the right-hand-side is concentrated in degree $p-1+1$ by \refm{NNN}{p-1} we see that so is the left-hand-side.
	We next apply $\sHom_R(S,-)$ to the first kind. It gives an exact sequence
	\[ \xymatrix{ \sHom_R(S,\Om^pM_j^\oplus)\ar[r]&\sHom_R(S,\Om^pN_{i,cn+j})\ar[r]&\sHom_R(S,\Om^{p-1}N_{i,cn+j-1}) }, \]
	in which the left term is concentrated in degree $p$ by assumption \refm{MMM}{p}, and the right term is concentrated in degree $p-1$ by \refm{NNN}{p-1}. It follows that the middle term is concentrated in degree $p-1$ and $p$. We claim that $\sHom^\Z_R(S,\Om^pN_{i,cn+j}(p-1))=0$, which will complete the induction step. Since $S$ is generated in degree $0$ and $\Om^pN_{i,cn+j}(p-1)$ in degree $1$ by \ref{ind}(\ref{genN}), the graded vector space $\Hom_R(S,\Om^pN_{i,cn+j}(p-1))$ concentrates in degree $\geq1$, hence $\sHom^\Z_R(S,\Om^pN_{i,cn+j}(p-1))=0$.
	
	Next we assume \refm{NNN}{p} and prove \refm{MMM}{p+1} by induction on $i$. By the last short exact sequence of the $p$-th syzygy of the almost split sequence at $M_1$, we see that $\Om^{p+1}M_1=\Om^pN_{1,d-2}(-1)$, thus $\sHom_R(S,\Om^{p+1}M_1)=\sHom_R(S,\Om^pN_{1,d-2})(-1)$ is concentrated in degree $p+1$ by \refm{NNN}{p}. Assume $2\leq i\leq n-1$ and consider the same last short exact sequence
	\[ \xymatrix{ 0\ar[r]& \Om^pN_{i,d-2}\ar[r]&\Om^{p}M_{i-1}^\oplus\ar[r]&\Om^{p}M_i\ar[r]&0 } \]
	in the almost split sequence at $\Om^pM_i$. Applying $\sHom_R(S,-)$ gives an exact sequence
	\[ \xymatrix{ \sHom_R(S,\Om^{p+1}M_{i-1}^\oplus)\ar[r]&\sHom_R(S,\Om^{p+1}M_{i})\ar[r]&\sHom_R(S,\Om^{p}N_{i,d-2}) }, \]
	in which the left term is concentrated in degree $p+1$ by induction hypothesis, and the right term is concentrated in degree $p$ by \refm{NNN}{p}. It remains to show $\sHom_R^\Z(S,\Om^{p+1}M_i(p))=0$. Similarly as above, the fact that $S$ is generated in degree $0$ and $\Om^{p+1}M_i(p)$ in degree $1$ shows $\Hom_R(S,\Om^{p+1}M_i(p))$ is concentrated in degree $\geq1$, so $\sHom^\Z_R(S,\Om^{p+1}M_i(p))=0$.
\end{proof}

Now we are ready to prove our main result \ref{VER} of this section.
\begin{proof}[Proof of \ref{VER}]
	We have seen in \ref{seisei} that $T=S\oplus \Om S(1)\oplus \cdots\oplus \Om^{a-1}S(a-1)$ generates $\sg R$. Also we known that $\sEnd^\Z_R(T)$ has finite global dimension since it is a triangular matrix algebra with diagonal entries $\sEnd_R(S)$ which has finite global dimension.
	To show that $T$ is a tilting object with $(d-a-1)$-representation infinite endomorphism ring, we have to prove
	\begin{itemize}
		\item $\sHom_R^\Z(T,T[p])=0$ for all $p\neq0$,
		\item $\sHom_R^\Z(T,\nu_{d-a-1}^{-q}T[p])=0$ for all $q\geq0$ and $p\neq0$, and
		\item $\sHom_R^\Z(\nu T,T[d-a-1+p])=0$ for all $p>0$ (so that $\gd\sEnd_R^\Z(T)\leq d-a-1$).
	\end{itemize}
	Computing the Serre functors shows that we have $\nu_{d-a-1}^{-1}=(a)[-a]$ on $\sg^\Z\!R$. In view of the formula $T=S\oplus \Om S(1)\oplus \cdots\oplus \Om^{a-1}S(a-1)$, it is enough to prove $\sHom_R^\Z(S,\Om^qS(q)[p])=0$ for all $q>-a$ and $p\neq0$. Note that this is the degree $q$ part of $\sHom_R(S,\Om^{q-p}S)$.
	
	If $p-q\leq0$, then we see $\sHom_R(S,\Om^{q-p}S)_q=0$ unless $p=0$ by \ref{conc}(\ref{pS}).
	
	If $1\leq p-q<d-1$, then $\sHom_R(S,\Om^{q-p}S)=\Ext_R^{p-q}(S,S)=0$ since $S\in\sg R$ is $(d-1)$-cluster tilting.
	
	If $d-1\leq p-q$, then by Serre duality we have $D\sHom_R^\Z(S,\Om^{q-p}S(q))=\sHom_R^\Z(\Om^{q-p}S(q),S(-a)[d-1])=\sHom_R(S,\Om^{p-q-d+1}S)_{-a-q}$. Since $p-q-d+1\geq0$ we see by \ref{conc}(\ref{pS}) that this is $0$ unless $p-q-d+1=-a-q$, or $p=d-a-1$. But it is impossible to have $q>-a$, $p-q\geq d-1$, and $p=d-a-1$ at the same time.
	
	Finally, since the $a$-th root of $\nu_{d-a-1}$ on $\Db(\mod A)$ is given by the automorphism corresponding to $\Om(1)=(1)[-1]$ on $\sg^\Z\!R$, its strictness of is a consequence of the fact that $T$ is obtained as an $\Om(1)$-orbit.
\end{proof}

\begin{Rem}
In the course of the proof of \ref{VER}, we have directly proved that the endomorphism ring $A$ is $(d-a-1)$-representation infinite. When $a=1$, we could have proved this using \ref{-1RI}, see \ref{trunc}.
\end{Rem}

As a consequence of $(d-a-1)$-represenation infiniteness of $A$, we deduce that we obtain (twisted) Calabi-Yau algebras from Veronese subrings. To state the result let us recall some relevant notions.
\begin{Def}[{cf. \cite{Gi,Ke11}}]
	A graded algebra $C=\bigoplus_{i\in\Z}C_i$ is {\it twisted $m$-Calabi-Yau} (CY) and {\it has Gorenstein parameter $l$} if
	\begin{itemize}
		\item $C$ is smooth, that is, $C\in\per^\Z\!C^e$,
		\item there is an isomorphism $\RHom_{C^e}(C,C^e)[m]\simeq {}_1C_\a(l)$	in $\D(\Mod^\Z\!C^e)$ for some homogeneous automorphism $\a$ of $C$.
	\end{itemize}
\end{Def}
When $C$ is positively graded and each $C_i$ is finite dimensional, the above CY property is equivalent to the following Artin-Schelter type condition by \cite[5.2, 5.15]{RR}, where $S=C/J_C$ is the quotient of $S$ by its graded Jacobson radical $J_C$;
\begin{itemize}
	\item $S$ is separable over $k$,
	\item $\gd(\Mod^\Z\!C)\leq m$,
	\item $\RHom_C(S,C)[m]\simeq S(l)$ in $\rD(\Mod^\Z\!C^\op)$.
\end{itemize}

Let us note one observation on (twisted) CY properties of a graded algebra and of its covering. Let $C=\bigoplus_{i\in\Z}C_i$ be a graded algebra and $a>0$. Then its $a$-th {\it covering} (or the {\it quasi-Veronese algebra}) is the graded algebra
\[ C^{[a]}=\bigoplus_{i\in\Z}C^{[a]}_i,\qquad C^{[a]}_i=\begin{pmatrix}C_{ia}&C_{ia-1}&\cdots&C_{ia-a+1}\\C_{ia+1}&C_{ia}&\cdots&C_{ia-a+2}\\ \vdots&\vdots&\ddots&\vdots\\ C_{ia+a-1}&C_{ia+a-2}&\cdots&C_{ia}\end{pmatrix} \]
with the natural multiplication. Then we have an equivalence of categories with compatible degree shifts
\[ \xymatrix{ \Mod^\Z\!C\ar[r]^-\simeq\ar@(dl,dr)[]_-{(a)}&\Mod^\Z\!C^{[a]}\ar@(dl,dr)[]_-{(1)}, & M=\disoplus_{i\in\Z}M_i\ar@{|->}[r]&\disoplus_{i\in\Z}(M_{ia},M_{ia-1},\ldots,M_{ia-a+1})}, \]
and similarly for $\Mod^\Z\!C^\op\xsimeq\Mod^\Z(C^{[a]})^\op$.
\begin{Lem}\label{qv}
	Let $C=\bigoplus_{i\geq0}C_i$ be a positively graded algebra such that each $C_i$ is finite dimensional, and let $a>0$ be an integer. Then $C$ is twisted $m$-CY of Gorenstein parameter $al$ if and only if $C^{[a]}$ is twisted $m$-CY of Gorenstein parameter $l$.
\end{Lem}
\begin{proof}
	We show the equivalence of the Artin-Schelter conditions. Since $T:=C^{[a]}/J_{C^{[a]}}$ is some product of $S:=C/J_C$, the separability is clearly equivalent. Secondly, by the equivalence $\Mod^\Z\!C\xsimeq\Mod^\Z\!C^{[a]}$, the finiteness of global dimension is also equivalent. Finally we consider the extensions. For $0\leq i\leq a-1$, let $e_i$ be the idempotent of $C^{[a]}$ corresponding to the $i$-th factor. Then for each $0\leq i\leq a-1$ the equivalence $\Mod^\Z\!C\xsimeq\Mod^\Z\!C^{[a]}$ takes $S(i)$ to $e_iT$ and $C(i)$ to $e_iC^{[a]}$, so we have $\RHom_C^\Z(S,C(i))\xsimeq\RHom_{C^{[a]}}^\Z(e_0T,e_iC^{[a]})$.
	Therefore $\RHom_{C^{[a]}}(T,C^{[a]})[m]\simeq T(l)$ implies, by multiplying $e_0$ from both sides, $\RHom_C(S,C)[m]\simeq S(al)$. The converse is similar; if we have $\RHom_C(S,C)[m]\simeq S(al)$ then $\RHom_C(\bigoplus_{i=0}^{a-1}S(i),\bigoplus_{i=0}^{a-1}C(i))[m]\simeq\End_C^\Z(\bigoplus_{i=0}^{a-1}S(i))(al)$, thus the equivalence of graded module categories gives $\RHom_{C^{[a]}}(T,C^{[a]})\simeq T(l)$.
\end{proof}

An important category associated to a twisted CY algebra is
\[ \qper^\Z\!C:=\per^\Z\!C/\Db(\fl^\Z\!C), \]
the Verdier quotient of the graded prefect derived category by the derived category of finite length graded modules. When $C$ is Noetherian (or more generally graded coherent), this is exactly the derived category of Serre quotient $\mod^\Z\!C/\fl^\Z\!C$, the non-commutative projective scheme in the sense of \cite{AZ}.


\begin{Cor}\label{qgr}
	Let $a\geq1$ and $n\geq2$ be integers, $d=na$, and $R=k[x_1,\ldots,x_d]^{(n)}$ the Veronese subring. Put
	\[ \G=\bigoplus_{i\geq0}\sHom_R(S,\Om^iS). \]
	\begin{enumerate}
		\item $\G$ is twisted $(d-a)$-CY algebra of Gorenstein parameter $a$.
		\item There exists an equivalence
		\[ \xymatrix{ \qper^\Z\!\G \ar@{-}[r]^-\simeq&\sg^\Z\!R } \]
		restricting to their canonical $(d-1)$-cluster tilting subcategories
		\[ \xymatrix{ \add\{\G(i)[i]\mid i\in\Z\}\ar@{-}[r]^-\simeq& \add\{S(i)\mid i\in\Z\} }. \]
	\end{enumerate}	
\end{Cor}
\begin{proof}
	(1)  Now first that by \ref{conc}(\ref{pS}) we can write $\G=\bigoplus_{i\geq0}\sHom_R^\Z(S,\Om^iS(i))$ as the Ext-algebra in the graded stable category $\smod^\Z\!R$. Since $A=\sEnd^\Z_R(T)$ is $(d-a-1)$-representation infinite, its $(d-a)$-preprojective algebra given by $\Pi=\bigoplus_{l\geq0}\Hom_A(A,\nu_{d-a-1}^{-l}A)$ is $(d-a)$-CY of Gorenstein parameter $1$ (\cite[4.8]{Ke11}, see also \cite[4.36]{HIO}\cite[3.3]{AIR}). By the compatibility of the autoequivalences $\nu_{d-a-1}^{-1}$ and $(a)[-a]$ under the equivalence $\Db(\mod A)\simeq\sg^\Z\!R$, we have $\Pi=\bigoplus_{i\geq0}\sHom_R^\Z(T,\Om^{ia}T(ia))$ as graded algebras. In view of the formula $T=S\oplus \Om S(1)\oplus\cdots\oplus\Om^{a-1}S(a-1)$, we see that $\Pi$ is the $a$-th covering of $\G$. We conclude by \ref{qv} that $\G$ is twisted $d$-CY of Gorenstein parameter $a$.\\
	(2)  The subcategories are indeed cluster tilting by \cite[2.5]{Iy07a} for $\sg^\Z\!R$, and by \cite[4.6]{Ha3} for $\qper^\Z\!\G$. By \cite[4.12]{MM} (see also \cite[4.9]{Ha3}) we have a triangle equivalence $\qper^\Z\!\G\simeq\Db(\mod A)$ taking $\G\oplus\G(1)\oplus\cdots\oplus\G(a-1)$ to $A$. Composing with $\Db(\mod A)\simeq\sg^\Z\!R$ given in \ref{VER} we obtain $\qper^\Z\!\G\simeq\sg^\Z\!R$ with compatible autoequivalences $(1)\leftrightarrow(1)[-1]$, which takes $\G$ to $S$. We deduce that the equivalence restricts to the $(d-1)$-cluster tilting subcategories.
\end{proof}

\subsection{Quivers and relations}
Let us further describe the algebra $\G$ and $A$ by quivers with relations.
We first observe that $\G$ is generated in degree $1$. Recall that the last few terms of the almost split sequence are given, for example, by
\[ \xymatrix@R=2mm@!C=10mm{
	\Wedge^3V\otimes M_{n-2}(-1)\ar[rr]\ar[dr]&&\Wedge^2V\otimes M_{n-1}(-1)\ar[rr]\ar[dr]&&V \otimes R\ar[dr]&\\
	&N_{1,d-3}(-1)\ar[ur]&&\Om M_1\ar[ur]&& M_1
} \]
where $V$ is the vector space with basis $\{x_1,\ldots,x_d\}$. The similar sequences exist for the syzygies (\ref{ind}(\ref{syz})).
\begin{Lem}\label{deg1}
There exist following exact sequences in $\mod R$. When $n>2$,
	\begin{enumerate}
		\item\label{for i>2} $\xymatrix{ \sHom_R(S,\Wedge^2V\otimes\Om^pM_{i-2})\ar[r]&\sHom_R(S,V\otimes\Om^pM_{i-1})\ar[r]&\sHom_R(S,\Om^pM_i) }$ for $i>2$. Moreover, the last map is surjective if $p>0$ and has a simple cokernel if $p=0$.
		\item\label{for i=1} $\xymatrix{ \sHom_R(S,\Wedge^3V\otimes\Om^pM_{n-2})\ar[r]&\sHom_R(S,\Wedge^2V\otimes\Om^pM_{n-1})\ar[r]&\sHom_R(S,\Om^{p+1}M_1)\ar[r]&0}$,
		\item\label{for i=2} $\xymatrix{ \sHom_R(S,\Wedge^3V\otimes\Om^{p}M_{n-1})\ar[r]&\sHom_R(S,V\otimes\Om^{p+1}M_{1})\ar[r]&\sHom_R(S,\Om^{p+1}M_2)\ar[r]&0 }$.
		\suspend{enumerate}
		When $n=2$,
		\resume{enumerate}
		\item\label{for n=2} $\xymatrix@C=6mm{ \sHom_R(M,\Wedge^4V\otimes\Om^{p}M)\ar[r]&\sHom_R(M,\Wedge^2V\otimes\Om^{p+1}M)\ar[r]&\sHom_R(M,\Om^{p+2}M)\ar[r]&0 }$, where $M=M_1$.
	\end{enumerate}
Consequently, $\G$ is generated in degree $1$ over its degree $0$ part.
\end{Lem}
%
\begin{proof}
	The last assertion follows inductively from surjectivity of right maps. We only prove (\ref{for i=1}) using the almost split sequence at $M_1$, the others are shown similarly.
	
	We first perform a computation in $\md^\Z\!R$. Applying $\sHom_R(S,-)$ to the almost split sequence ending at $M_1$ yields an exact sequence
	\[ \xymatrix@R=2mm{
		&&\sHom_R(S,\Om^{p+2}M_1(1))\ar[r]&\\
		\sHom_R(S,\Om^pN_{1,d-3})\ar[r]&\sHom_R(S,\Wedge^2V\otimes\Om^pM_{n-1})\ar[r]&\sHom_R(S,\Om^{p+1}M_1(1))\ar[r]&\\
		\sHom_R(S,\Om^{p-1}N_{1,d-3}) } \]
	in $\md^\Z\!R$. Since its first term is concentrated in degree $p+1$, next three terms in degree $p$, and the last term in degree $p-1$ by \ref{conc}, the connecting morphisms are $0$, so the middle terms give a short exact sequence. Also we see similarly that the map $\sHom_R(S,\Wedge^3V\otimes\Om^pM_{n-2})\to\sHom_R(S,\Om^pN_{1,d-3})$ is surjective, which gives the desired exact sequence.
\end{proof}

The sequences above lead to the following description of $\G$.
\begin{Thm}\label{Gamma}
	The algebra $\G=\bigoplus_{i\geq0}\sHom_R(S,\Om^iS)$ is presented by the quiver $Q$ with
	\begin{enumerate}
		\renewcommand{\labelenumi}{(\alph{enumi})}
		\item vertices $\{M_i\mid 1\leq i\leq n-1\}\simeq\{1,2,\ldots,n-1\}$,
		\item two types of arrows:
		\begin{itemize}
			\item $d$ arrows $x_s=x_s^{i}\colon M_i\to M_{i+1}$, $1\leq s\leq d$ for $1\leq i<n-1$,
			\item $\Binom{d}{2}$ arrows $x_{st}\colon M_{n-1}\to M_1$, $1\leq s<t\leq d$,
		\end{itemize}
		\item relations:
		\begin{itemize}
			\item $\Binom{d}{2}$ relations $x^{i+1}_sx^{i}_t=x^{i+1}_tx^{i}_s$, $1\leq s<t\leq d$ for each $1\leq i<n-1$,
			\item $\Binom{d}{3}$ relations $x^{1}_sx_{tu}-x^{1}_tx_{su}+x^{1}_ux_{st}=0$, $1\leq s<t<u\leq d$, and similarly $x_{st}x^{n-2}_u-x_{su}x^{n-2}_t+x_{tu}x^{n-2}_s=0$, $1\leq s<t<u\leq d$,
			\item if $n=2$, $\Binom{d}{4}$ relations $x_{uv}x_{st}-x_{tv}x_{su}+x_{tu}x_{sv}+x_{sv}x_{tu}-x_{su}x_{tv}+x_{st}x_{uv}=0$ for $1\leq s<t<u<v\leq d$.
		\end{itemize}
	\end{enumerate}
\end{Thm}
\begin{proof}
	Since $\G$ is generated in degree $1$ by \ref{deg1}, the quiver of $\G$ is obtained from the quiver of $\G_0$ by adding some degree $1$ arrows. We know that the degree $0$ part $\sEnd_R(S)$ is presented by the quiver with vertives $\{M_1,\ldots,M_{n-1}\}$, $d$ arrows $\{x_1,\ldots,x_d\}$ from $M_{i-1}$ to $M_{i}$ for each $1<i\leq n-1$.
	\[ \xymatrix{ M_1\ar@2[r]&M_2\ar@2[r]&\cdots\ar@2[r]& M_{n-1} } \]
	
	We next consider degree $1$ arrows. By \ref{deg1}(\ref{for i>2})(\ref{for i=2}), any morphism $S\to \Om M_i$ factors through the map $\Om M_{i-1}\otimes V\to \Om M_i$ whenever $i>1$, so any arrow of degree $1$ ends at $M_1$. Similarly, using the first several terms of almost split sequences we see that any degree $1$ arrows starts at $M_{n-1}$. Now \ref{deg1}(\ref{for i>2})(\ref{for n=2}) yields an isomorphism $\sHom_R(M_{n-1},\Wedge^2V\otimes M_{n-1})\xsimeq\sHom_R(M_{n-1},\Om M_1)$, which shows the number of arrows from $M_{n-1}$ to $M_1$ is $\dim\Wedge^2V=\Binom{d}{2}$. We label them by $\{x_{st}\mid 1\leq s<t\leq d\}$.
	\[ \xymatrix{ \cdots\ar@2[r]^-d&M_{n-1}\ar@3[r]^-{\binom{d}{2}}&M_1\ar@2[r]^-d&M_2\ar@2[r]^-d&\cdots\ar@2[r]^-d& M_{n-1}\ar@3[r]^-{\binom{d}{2}}&M_1\ar@2[r]^-d&\cdots } \]
	
	We go on to determine the relations. 
	First assume $n>2$.	The exact sequence in \ref{deg1}(\ref{for i>2}) shows that all relations ending at $M_i$ start at $M_{i-2}$, and that they are reflected as the following composite being zero, see \cite[3.3]{BIRSm}.
	\[ \xymatrix{\sHom_R(M_{i-2},\Wedge^2V\otimes M_{i-2})\ar[r]&\sHom_R(M_{i-2},V\otimes M_{i-1})\ar[r]&\sHom_R(M_{i-2},M_i) } \]
	An element $x_s\wedge x_t\in\Wedge^2V\simeq\sHom_R(M_{i-2},M_{i-2}\otimes\Wedge^2V)$ is mapped to $x_t\otimes x_s-x_s\otimes x_t$ in $\sHom_R(M_{i-2},V\otimes M_{i-1})\simeq V\otimes V$, which shows that we have commutativity relation from $M_{i-2}$ to $M_i$.
	
	Similarly consider the sequence in \ref{deg1}(\ref{for i=1}) which yields the relations from $M_{n-1}$ to $M_1$. In the sequence
	\[ \xymatrix@R=2mm{
		\sHom_R(M_{n-2},\Wedge^3V\otimes M_{n-2})\ar@{=}[d]\ar[r]&\sHom_R(M_{n-2},\Wedge^2V\otimes M_{n-1})\ar@{=}[d]\ar[r]&\sHom_R(M_{n-2},\Om M_1) \\
		\Wedge^3V& \Wedge^2V\otimes V&&, } \]
	the first map is given by $x_s\wedge x_t\wedge x_u\mapsto (x_t\wedge x_u)\otimes x_s-(x_s\wedge x_u)\otimes x_t+(x_s\wedge x_t)\otimes x_u$ for $1\leq s<t<u\leq d$. In view of our identification of the arrows $M_{n-1}\to M_1$ with the basis of $\sHom_R(M_{n-1},\Om M_1)\simeq\Wedge^2 V$, we see that the relations from $M_{n-2}$ to $M_1$ are given by $x_{tu}x_s-s_{su}x_t+s_{st}x_u=0$. The same holds for the relations from $M_{n-1}$ to $M_2$.
	
	Next assume $n=2$. In this case the exact sequence in \ref{deg1}(\ref{for n=2}) again shows the relations are quadratic, and gives an exact sequence below.
	\[ \xymatrix@R=2mm{
		\sHom_R(M,\Wedge^4V\otimes M)\ar[r]\ar@{=}[d]&\sHom_R(M,\Wedge^2V\otimes\Om M)\ar[r]\ar@{=}[d]&\sHom_R(M,\Om^{2}M) \\
		\Wedge^4V&\Wedge^2V\otimes\Wedge^2V } \]
	Let us compute the first map in terms of the second row. Recall that the above exact sequence is obtained from the almost split sequence
	\[ \xymatrix@R=3mm@!C=7mm{
		\Wedge^4V\otimes M\ar[rr]\ar[dr]_-f&&\Wedge^3V\otimes R\ar[rr]\ar[dr]&&\Wedge^2V\otimes M\ar[dr]\ar[rr]&&V\otimes R\ar[dr]&\\
		&\Om N\ar[ur]&&N\ar[ur]_-g&&\Om M\ar[ur]&& M }\]
	as the composite
	\[ \xymatrix{ \sHom_R(M,\Wedge^4V\otimes M)\ar[r]^-{f\cdot}&\sHom_R(M,\Om N)\ar[r]^-{\Om g\cdot}&\sHom_R(M,\Wedge^2V\otimes\Om M) }, \]
	and the map $\Om g$ is computed by the following commutative diagram
	\[ \xymatrix{
		0\ar[r]&\Om N\ar[r]\ar[d]_-{\Om g}&\Wedge^3V\otimes R\ar[r]\ar[d]& N\ar[r]\ar[d]^-g& 0\\
		0\ar[r]&\Wedge^2V\otimes\Om M\ar[r]&\Wedge^2V\otimes V\otimes R\ar[r]&\Wedge^2V\otimes M\ar[r]& 0 } \]
	whose middle map is the one induced by $\Wedge^3V\to \Wedge^2V\otimes V$, $x_s\wedge x_t\wedge x_u\mapsto(x_t\wedge x_u)\otimes x_s-(x_s\wedge x_u)\otimes x_t+(x_s\wedge x_t)\otimes x_u$. Therefore the composite $\Wedge^4V\otimes M\xrightarrow{f}\Om N\xrightarrow{\Om g}\Wedge^2V\otimes\Om M\hookrightarrow\Wedge^2V\otimes V\otimes R$ which equals $\Wedge^4V\otimes M\to\Wedge^3V\otimes R \to\Wedge^2V\otimes V\otimes R$ is described by
	\begin{equation*}
		\begin{aligned}
			(x_s\wedge x_t\wedge x_u\wedge x_v)\otimes m
			\mapsto\,\,&(x_t\wedge x_u\wedge x_v)\otimes x_sm-(x_s\wedge x_u\wedge x_v)\otimes x_tm\\
			+&(x_s\wedge x_t\wedge x_v)\otimes x_um-(x_s\wedge x_t\wedge x_u)\otimes x_vm\\
			\mapsto\,\, & (x_u\wedge x_v)\otimes x_t\otimes x_sm -(x_t\wedge x_v)\otimes x_u\otimes x_sm+(x_t\wedge x_u)\otimes x_v\otimes x_sm\\
			-&(x_u\wedge x_v)\otimes x_s\otimes x_tm +(x_s\wedge x_v)\otimes x_u\otimes x_tm -(x_s\wedge x_u)\otimes x_v\otimes x_tm\\
			+&(x_t\wedge x_v)\otimes x_s\otimes x_um-(x_s\wedge x_v)\otimes x_t\otimes x_um +(x_s\wedge x_t)\otimes x_v\otimes x_um\\
			-&(x_t\wedge x_u)\otimes x_s\otimes x_vm+(x_s\wedge x_u)\otimes x_t\otimes x_vm-(x_s\wedge x_t)\otimes x_u\otimes x_vm\\
			=\,\, & (x_u\wedge x_v)\otimes(x_t\otimes x_sm-x_s\otimes x_tm)-(x_t\wedge x_v)\otimes (x_u\otimes x_sm-x_s\otimes x_um)\\
			+&(x_t\wedge x_u)\otimes (x_v\otimes x_sm-x_s\otimes x_vm)+(x_s\wedge x_v)\otimes(x_u\otimes x_tm-x_t\otimes x_um)\\
			-&(x_s\wedge x_u)\otimes (x_v\otimes x_tm-x_t\otimes x_vm)+(x_s\wedge x_t)\otimes (x_v\otimes x_um-x_u\otimes x_vm).
		\end{aligned}
	\end{equation*}	
	Recalling that the isomorphism $\Wedge^2V\simeq\sHom_R(M,\Om M)$ is given by $x_s\wedge x_t\mapsto (m\mapsto x_t\otimes x_sm-x_s\otimes x_tm)$ under the inclusion $\Om M\subset V\otimes R$, we see that the map in question $\Wedge^4V\to\Wedge^2V\otimes\Wedge^2V$ can be written as
	\begin{equation*}
		\begin{aligned}
			x_s\wedge x_t\wedge x_u\wedge x_v
			\mapsto\,\, &(x_u\wedge x_v)\otimes (x_s\wedge x_t)-(x_t\wedge x_v)\otimes (x_s\wedge x_u)+(x_t\wedge x_u)\otimes (x_s\wedge x_v)\\
			+&(x_s\wedge x_v)\otimes (x_t\wedge x_u)-(x_s\wedge x_u)\otimes (x_t\wedge x_v)+(x_s\wedge x_t)\otimes (x_u\wedge x_v).
		\end{aligned}
	\end{equation*}
	We therefore conclude that the relations are as asserted.
\end{proof}

\begin{Cor}
	The endomorphism algebra $A=\sEnd_R^\Z(T)$ given in \ref{VER} is presented by the quiver consisting of
	\begin{enumerate}
		\renewcommand{\labelenumi}{(\alph{enumi})}
		\item vertices: $\{\Om^pM_i\mid 1\leq i\leq n-1, 0\leq p\leq a-1\}\simeq\{1,\ldots,n-1\}\times\{0,1,\ldots,a-1\}$,
		\item arrows: 
		\begin{itemize}
			\item $d$-arrows $\Om^pM_i\to\Om^pM_{i+1}$ for each $1\leq i<n-1$ and $0\leq p\leq a-1$,
			\item $\Binom{d}{2}$ arrows $\Om^pM_{n-1}\to\Om^{p+1}M_1$ for each $0\leq p<a-1$,
		\end{itemize}
	\end{enumerate}
	with the induced relations.
\end{Cor}
\begin{proof}
	This is because $A=\End_\G^\Z(\G\oplus\cdots\oplus\G(a-1))$.
\end{proof}

Let us give some examples of the algebra $\G$ and $A$.
\begin{Ex}
	The first example is $a=1$, that is, $R=k[x_1,\ldots,x_d]^{(d)}$. The algebra $A$ in this case is presented by the linear quiver below with $d$-fold arrows and commutativity relations.
	\[ \xymatrix{ M_1\ar@2[r]^-d&M_2\ar@2[r]^-d&\cdots\ar@2[r]^-d&M_{n-1} } \]
	This appears also in \cite[5.5]{AIR}, which turns out by \ref{VER}(2) to be $(d-2)$-representation infinite. In this case the algebra $\G$ is nothing but the $(d-1)$-preprojective algebra of $A$, which by \ref{Gamma} is presented by the cyclic quiver
	\[ \xymatrix{ \cdots\ar@2[r]&M_{n-1}\ar@3[r]&M_1\ar@2[r]&M_2\ar@2[r]&\cdots\ar@2[r]&M_{n-1}\ar@3[r]&M_1\ar@2[r]&\cdots } \]
	with suitable relations. We list some specific examples for small $d$, where the dotted arrows for $A$ show the relations, and the figures show the number of arrows or relations.
	\[
	\begin{tabular}{c|ccc}
		$d$ & $A$ & $\G=\Pi_{d-1}(A)$ &relations \\
		\hline
		$2$ & $\circ$ & $\xymatrix@R=3mm{{}\\ \circ\ar@(ul,ur)[]\ar@{}[u]}$ \\
		\hline
		$3$ & $\xymatrix{M_1\ar[r]^3&M_2}$ & $\xymatrix{M_1\ar@/^3pt/[r]^-3&M_2\ar@/^3pt/[l]^-3}$ & $\begin{aligned} &x_{23}x_1-x_{13}x_2+x_{12}x_3=0 \\&x_1x_{23}-x_2x_{13}+x_3x_{12}=0\end{aligned}$ \\
		\hline
		$4$ & $\xymatrix@R=3mm@C=3mm{&M_2\ar[dr]^-{4}\\ M_1\ar[ur]^-{4}&&M_3\ar@{-->}[ll]^-{6} }$  &$\xymatrix@R=3mm@C=3mm{&M_2\ar[dr]^-{4}\\ M_1\ar[ur]^-{4}&&M_3\ar[ll]^-{6} }$& \\
		\hline
		$5$ & $\xymatrix{M_2\ar[r]^-5&M_3\ar[d]^-5\ar@{-->}[dl]^-{10}\\M_1\ar[u]^-5&M_4\ar@{-->}[ul]^-{10}}$&$\xymatrix{M_2\ar[r]^-5&M_3\ar[d]^-5\\M_1\ar[u]^-5&M_4\ar[l]_-{10}}$
	\end{tabular}
	\]
\end{Ex}

\begin{Ex}
	Let us look at the case $n=2$, thus $d=2a$ and $R=k[x_1,\ldots,x_d]^{(2)}$. We obtain a family of $(d/2-1)$-representation infinite algebras $A$ and twisted $d/2$-CY algebra $\G$ below of Gorenstein parameter $d/2$ with $\G_0=k$, generated by $\Binom{d}{2}$ elements and $\Binom{d}{4}$ relations.
	\[ \G=\frac{k\left\langle x_{st}\mid 1\leq s<t\leq d\right\rangle}{ (x_{uv}x_{st}-x_{tv}x_{su}+x_{tu}x_{sv}+x_{sv}x_{tu}-x_{su}x_{tv}+x_{st}x_{uv}=0\mid 1\leq s<t<u<v\leq d)} \]
	Again we give some details for small $d$.
	\[
	\begin{tabular}{c|cc}
		$d$ & $\G$ & $A$\\
		\hline
		$2$ & $k[x]$ & $\circ$ \\
		\hline
		$4$ & $\Frac{k\left\langle x_1,x_2,x_3,x_4,x_5,x_6\right\rangle}{(x_6x_1-x_5x_2+x_4x_3+x_4x_3-x_2x_5+x_1x_6)}$ & $\xymatrix{\circ\ar[r]^6&\circ}$ \\
		\hline
		$6$ && $\xymatrix@R=3mm@C=3mm{&\circ\ar[dr]^-{15}\\ \circ\ar[ur]^-{15}&&\circ\ar@{-->}[ll]_-{15} }$ \\
		\hline
		$8$&& $\xymatrix@R=4mm@C=4mm{&\circ\ar[r]^-{28}&\circ\ar[dr]^-{28}\ar@{-->}[dll]^-{70}&\\ \circ\ar[ur]^-{28}&&& \circ\ar@{-->}[ull]^-{70}}$
	\end{tabular}
	\]
\end{Ex}

\begin{Ex}
	Let us include the case $n=3$ on the left, and $n=4$ on the right.
	\[
	\begin{tabular}{c|cc}
		$d$ & $\G$ & $A$\\
		\hline
		$3$ & $\xymatrix{M_1\ar@/^3pt/[r]^-3&M_2\ar@/^3pt/[l]^-3}$ & $\xymatrix{M_1\ar[r]^-3&M_2}$ \\
		\hline
		$6$ & $\xymatrix{M_1\ar@/^3pt/[r]^-6&M_2\ar@/^3pt/[l]^-{15}}$ & $\xymatrix{M_1\ar[r]^-6&M_2\ar[dl]_-{15}\\ \Om M_1\ar@{-->}[u]^-{20}\ar[r]^-6&\Om M_2\ar@{-->}[u]_-{20} }$ \\
		\hline
		$9$ &$\xymatrix{M_1\ar@/^3pt/[r]^-9&M_2\ar@/^3pt/[l]^-{36}}$ & $\xymatrix{M_1\ar[r]^-9&M_2\ar[dl]_-{36}\\ \Om M_1\ar[r]^-9\ar@{-->}[u]^-{84}&\Om M_2\ar[dl]_-{36}\ar@{-->}[u]_-{84}\\ \Om^2M_1\ar[r]^-9\ar@{-->}[u]^-{84}&\Om^2M_2\ar@{-->}[u]_-{84} }$
	\end{tabular}
	\qquad
	\begin{tabular}{c|cc}
		$d$ & $\G$ & $A$\\
		\hline
		$4$ & $\xymatrix{M_1\ar[r]^-4&M_2\ar[r]^-4&M_3\ar@/^10pt/[ll]^-6}$ & $\xymatrix{M_1\ar[r]^-4&M_2\ar[r]^-4&M_3}$ \\
		\hline
		$8$ & $\xymatrix{M_1\ar[r]^-8&M_2\ar[r]^-8&M_3\ar@/^10pt/[ll]^-{28}}$ & $\xymatrix{M_1\ar[r]^-8&M_2\ar[r]^-8&M_3\ar[dll]_-{28}\\ \Om M_1\ar@{-->}[ur]^-{56}\ar[r]^-8&\Om M_2\ar[r]^-8\ar@{-->}[ur]_-{56}&\Om M_3 }$ \\
		\hline
		$12$ &$\xymatrix{M_1\ar[r]^-{12}&M_2\ar[r]^-{12}&M_3\ar@/^10pt/[ll]^-{66}}$& $\xymatrix{M_1\ar[r]^-{12}&M_2\ar[r]^-{12}&M_3\ar[dll]_-{66}\\ \Om M_1\ar@{-->}[ur]^-{220}\ar[r]^-{12}&\Om M_2\ar[r]^-{12}\ar@{-->}[ur]_-{220}&\Om M_3\ar[dll]_-{66}\\ \Om^2 M_1\ar@{-->}[ur]^-{220}\ar[r]^-{12}&\Om^2 M_2\ar[r]^-{12}\ar@{-->}[ur]_-{220}&\Om^2 M_3 }$
	\end{tabular}
	\]
\end{Ex}

\section{Homogeneous coordinate ring of a product of projective lines}
Recall that the {\it Segre product} of graded $k$-algebras $S=\bigoplus_{i\in\Z}S_i$ and $T=\bigoplus_{i\in\Z}T_i$ is the graded algebra
\[ S\seg T:=\bigoplus_{i\in\Z}S_i\otimes_k T_i. \]
If $S$ and $T$ are positively graded commutative algebras generated in degree $1$, then $S\seg T$ is the homogeneous coordinate ring of the product $X\times Y$ with respect to the polarization $\mathcal{O}(1,1)$, where $X=\Proj S$ and $Y=\Proj T$.
Similarly, if $M\in\mod^\Z\!S$ and $N\in\mod^\Z\!T$ are graded modules, then their Segre product is
\[ M\seg N:=\bigoplus_{i\in\Z}M_i\otimes_k N_i, \]
which is a graded module over $S\seg T$. If $\widetilde{M}\in\coh X$ and $\widetilde{N}\in\coh Y$ are the corresponding coherent sheaves, then $\widetilde{M\seg N}=\widetilde{M}\boxtimes\widetilde{N}$ in $\coh (X\times Y)$.

Let $n\geq1$ be an integer and consider the Segre product
\[ R=k[x_1,y_1]\seg k[x_2,y_2]\seg\cdots\seg k[x_n,y_n] \]
of $n$ copies of the polynomial ring in $2$ variables, where we give the standard $\Z$-grading to each of them. This is a homogeneous coordinate ring of the $n$-fold product $X=\PP^1\times\cdots\times\PP^1$ of the projective line $\PP^1$, and $R$ is an $(n+1)$-dimensional Gorenstein isolated singularity of $a$-invariant $-2$.

This graded ring $R$ can also be described as an invariant ring in the following way. Let
\[ S=k[x_{i},y_{i}\mid 1\leq i\leq n], \]
which we view as a $\Z^n$-graded ring by $\deg x_{i}=\deg y_{i}=e_i$, where $e_i=(0,\ldots,1,\ldots,0)$ is the $i$-th unit vector. Then $R$ is a Veronese subring of $S$:
\[ R=S^{(1,1,\ldots,1)}, \]
that is, $R=\bigoplus_{i\in\Z}S_{(i,\ldots,i)}$ as graded rings.

Now we go on to fix further notations. For $v=(v_1,\ldots,v_n)\in\Z^n$ we write $v\geq0$ if $v_i\geq0$ for every $i$. Clearly, this makes $\Z^n$ into an ordered group by setting $v\geq v^\prime$ if $v-v^\prime\geq0$. We also write $0=(0,\ldots,0)$, $1=(1,\ldots,1)$, and $-1=(-1,\ldots,-1)$.
For a vector $v\in\Z^n$, we put $M_v:=\bigoplus_{i\in\Z}S_{v+(i,\ldots,i)}$ which is a graded $R$-module.

We show that this $R$ admits a tilting object with $(n-2)$-representation infinite endomorphism ring.
\begin{Thm}\label{P^1}
Let $R=k[x_1,y_1]\seg k[x_2,y_2]\seg\cdots\seg k[x_n,y_n]$ with the standard grading.
\begin{enumerate}
\item $T=\bigoplus_{0<v<1}M_v\in\sg^\Z\!R$ is a tilting object
\item $A:=\End_{\sg R}^\Z(T)$ is $(n-2)$-representation infinite with a root of $\nu_{n-2}$.
\end{enumerate}
Therefore, there is a commutative diagram of equivalences
\[ \xymatrix@R=5mm{
	\sg^\Z\!R\ar[r]\ar[d]^-\rsimeq&\sg^{\Z/2\Z}\!R\ar[r]\ar[d]^-\rsimeq&\sg R\ar[d]^-\rsimeq\\
	\Db(\mod A)\ar[r]&\rC_n(A)\ar[r]&\rC_{n}^{(1/2)}(A). } \]
\end{Thm}
We do not know in general if the above root of $\nu_{n-2}$ can form a (strict) root pair for some tilting object.

For $n=3$ we have a positive answer, in fact, the above tilting object given above yields a strict root pair for $\nu_1$.
Note also that in this case we have that $A$ is hereditary, which therefore gives a new example of a Gorenstein ring of hereditary representation type \cite{Ha6}. In the statement below we write $M_{abc}$ for $M_{(a,b,c)}$.
\begin{Thm}\label{n=3}
Let $n=3$ so that $R=k[x_1,y_1]\seg k[x_2,y_2]\seg k[x_3,y_3]$, and let $T=\bigoplus_{0<v<1}M_v$ be the tilting object in \ref{P^1}. 
\begin{enumerate}
\item The functor $\Om(1)=(1)[-1]$ on $\sg^\Z\!R$ is a square root of $\nu_1^{-1}$.
\item We have $\Om M_{100}(1)=M_{011}$, $\Om M_{010}(1)=M_{101}$, and $\Om M_{001}(1)=M_{110}$. Therefore, $T=T_0\oplus \Om T_0(1)$ for $T_0=M_{100}\oplus M_{010}\oplus M_{001}$.
\item The endomorphism ring $\End_{\sg R}^\Z(T)$ is hereditary and has a strict root pair $(U,P)$ of $\nu_1^{-1}=\tau^{-1}$ for $U=\sHom_R^\Z(T,\Om T(1))$ and $P=\sHom_R^\Z(T,T_0)$.
\item The algebra $\End_{\sg R}^\Z(T)$ is presented by the following alternating $\widetilde{A_5}$ quiver $Q$ with double arrows between the vertices.
	\[ \xymatrix{
		M_{100}\ar@2[d]\ar@2[dr]&M_{010}\ar@2[dl]\ar@2[dr]&M_{001}\ar@2[d]\ar@2[dl]\\
		M_{110}&M_{101}&M_{011} } \]
\item There exists a commutative diagram of equivalences
	\[ \xymatrix@R=5mm{
	\sg^\Z\!R\ar[r]\ar[d]^-\rsimeq&\sg^{\Z/2\Z}\!R\ar[r]\ar[d]^-\rsimeq&\sg R\ar[d]^-\rsimeq\\
	\Db(\mod kQ)\ar[r]&\rC_3(kQ)\ar[r]&\rC_{3}^{(1/2)}(kQ). } \]
Therefore, $R$ is of hereditary representation type.
\item The above diagram restricts to equivalences
\[ \xymatrix@R=5mm{
	\add\{T_0(i)\mid i\in\Z\}\ar[d]^-\rsimeq\ar[r]&\add\{T_0(i)\mid i\in\Z/2\Z\}\ar[r]\ar[d]^-\rsimeq&\add T_0\ar[d]^-\rsimeq\\
	\add\{P\lotimes_{kQ}U^i[i]\mid i\in\Z\}\ar[r]&\add kQ\ar[r]&\add P } \]
which are $3$-cluster tilting subcategories.
\end{enumerate}
\end{Thm}

Our proof is based on Orlov's theorem (\cite{Or09}) and the reduction of higher hereditary algebras (Section \ref{red}).
\subsection{Derived Gorenstein parameters and Orlov's theorem}
Let $\La=\bigoplus_{i\geq0}\La_i$ be a positively graded Iwanaga-Gorenstein algebra such that $\La_0$ is a finite dimensional algebra of finite global dimension. By the Iwanaga-Gorenstein property, we have a duality
\[ (-)^\ast:=\RHom_\La(-,\La)\colon\Db(\mod\La)\leftrightarrow\Db(\mod\La^\op). \]
For a subset $I\subset\Z$ we denote by $\mod^I\!\La$ the full subcategory of $\mod^\Z\!\La$ formed by modules $X$ such that $X_i=0$ unless $i\in I$. When $I=\{i\in\Z\mid i\geq0\}$ we write $\mod^{\geq0}\!\La$ for $\mod^I\!\La$, with obvious variations.
We say that $\La$ has {\it derived Gorenstein parameter $a$} \cite{IYa2},see also \cite{Ha7}, if the above duality restricts to a one between
\[ \Db(\mod^0\!\La)\leftrightarrow\Db(\mod^{a}\!\La). \]
Similarly, we denote by $\proj^I\!\La$ the category of projective modules generated in degree $i\in I$, thus $\proj^I\!\La=\add\{\La(-i)\mid i\in I\}$.
Now consider the Serre quotient
\[ \qmod^\Z\!\La:=\mod^\Z\!\La/\fl^\Z\!\La \]
of the category of finitely generated graded modules by the category of finite length graded modules. When $\La$ is commutative and generated in degree $1$, this is nothing but the category of coherent sheaves over $\Proj\La$. For non-commutative $\La$, it is called a {non-commutative projective scheme} \cite{AZ}.
\begin{Thm}[\cite{Or09}]\label{Orlov}
Suppose that $\La$ has derived Gorenstein parameter $-p\in\Z$.
\begin{enumerate}
\item\label{qmod} The composite $\Db(\mod^{\geq0}\!\La)\cap\Db(\mod^{>-p}\!\La^\op)^\ast\subset\Db(\mod^\Z\!\La)\to\Db(\qmod^\Z\!\La)$ is an equivalence.
\item\label{sg} The composite $\Db(\mod^{\geq0}\!\La)\cap\Db(\mod^{>0}\!\La^\op)^\ast\subset\Db(\mod^\Z\!\La)\to\sg^\Z\!\La$ is an equivalence.
\item\label{sod} If $-p\leq0$, then there there exists a semi-orthogonal decomposition
\[ \Db(\qmod^\Z\!\La)=\sg^\Z\!\La\perp\Kb(\proj^{[0,p-1]}\!\La), \]
where we identify each category as a subcategory of $\Db(\mod^\Z\!\La)$.
\end{enumerate}
\end{Thm}

For our purpose, the following variation of \ref{Orlov}(\ref{sod}) will be suitable.
\begin{Prop}\label{ssodd}
Suppose that $\La$ has derived Gorenstein parameter $-p\leq0$. Then for each $0\leq q\leq p$, there is a semi-orthogonal decomposition
\[ \Db(\qmod^\Z\!\La)=\Kb(\proj^{[-q,-1]}\!\La)\perp\sg^\Z\!\La\perp\Kb(\proj^{[0,p-q-1]}\!\La), \]
where we identify $\Db(\mod^{\geq0}\!\La)\cap\Db(\mod^{>-p}\!\La^\op)^\ast\subset\Db(\mod^\Z\!\La)\to\Db(\qmod^\Z\!\La)$ and $\Db(\mod^{\geq 0}\!\La)\cap\Db(\mod^{>0}\!\La^\op)^\ast\xsimeq\sg^\Z\!\La$.
\end{Prop}
\begin{proof}
	This is an adaptation of the proof of \ref{Orlov}(\ref{sod}).
	Shifting \ref{Orlov}(\ref{qmod}) we have an equivalence
	\begin{equation}\label{sono1}
	\Db(\mod^{\geq -q}\!\La)\cap\Db(\mod^{>-p+q}\!\La^\op)^\ast\xsimeq\Db(\qmod^\Z\!\La).
	\end{equation}
	Consider the following categories with a semi-orthogonal decomposition $\Db(\mod^{\geq-q}\!\La)=\rK^{-,\rb}(\proj^{\geq-q}\!\La)=\Kb(\proj^{[-q,-1]}\!\La)\perp\rK^{-,\rb}(\proj^{\geq 0}\!\La)=\Kb(\proj^{[-q,-1]}\!\La)\perp\Db(\mod^{\geq 0}\!\La)$ which exists since $\La_0$ has finite global dimension. Noting that $\Kb(\proj^{[-q,-1]}\!\La)\subset\Db(\mod^{>-p+q}\!\La^\op)^\ast$, intersecting with $\Db(\mod^{>-p+q}\!\La^\op)^\ast$ yields a semi-orthogonal decomposition
	\begin{equation}\label{sono2}
	\Db(\mod^{\geq-q}\!\La)\cap\Db(\mod^{>-p+q}\!\La^\op)^\ast=\Kb(\proj^{[-q,-1]}\!\La)\perp(\Db(\mod^{\geq 0}\!\La)\cap\Db(\mod^{>-p+q}\!\La^\op)^\ast).
	\end{equation}
	Similarly, we have a semi-orthogonal decomposition $\Db(\mod^{>-p+q}\!\La)=\Kb(\proj^{[-p+q+1,0]}\!\La)\perp\Db(\mod^{>0}\!\La)$. Applying this to $\La^\op$ and using the duality $(-)^\ast$, it becomes $\Db(\mod^{>-p+q}\!\La^\op)^\ast=\Db(\mod^{>0}\!\La^\op)^\ast\perp\Kb(\proj^{[0,p-q-1]}\!\La)$. Then, intersecting with $\Db(\mod^{\geq 0}\!\La)$ yields
	\begin{equation}\label{sono3}
		\Db(\mod^{\geq 0}\!\La)\cap\Db(\mod^{>-p+q}\!\La^\op)^\ast=(\Db(\mod^{\geq 0}\!\La)\cap\Db(\mod^{>0}\!\La^\op)^\ast)\perp\Kb(\proj^{[0,p-q-1]}\!\La)
	\end{equation}
	by $\Kb(\proj^{[0,p-q-1]}\!\La)\subset\Db(\mod^{\geq 0}\!\La)$.
	We now obtain the desired assertion by \eqref{sono1}, \eqref{sono2}, and \eqref{sono3}.
\end{proof}

Applying \ref{Orlov} and \ref{ssodd} to our commutative ring $R$, we get the following consequence.
\begin{Prop}
Let $R=k[x_1,y_1]\seg\cdots\seg k[x_n,y_n]$ be the Segre product of $n$ copies of $k[x,y]$ with the standard grading on each of them.
\begin{enumerate}
\item There are triangle equivalences
\[
\begin{aligned}
	\Db(\mod^{\geq0}\!\La)\cap\Db(\mod^{>0}\!\La^\op)^\ast&\xsimeq&&\sg^\Z\!\La, \quad \text{and}\\
	\Db(\mod^{\geq-1}\!\La)\cap\Db(\mod^{>-1}\!\La^\op)^\ast&\xsimeq&&\Db(\qmod^\Z\!R).
\end{aligned}
\]
\item Under the above identifications, we have a semi-orthogonal decomposition
\[ \Db(\qmod^\Z\!R)=\Kb(\proj^{-1}\!R)\perp\sg^\Z\!R\perp\Kb(\proj^{0}\!R). \]
\end{enumerate}
\end{Prop}

\subsection{Reduction of $n$-representation infinite algebras and proof of \ref{P^1}}
To prove our result \ref{P^1}, we have to track the equivalences in \ref{Orlov}. 
Recall that for $v\in\Z^n$ we denote by $M_v=\bigoplus_{i\in\Z}S_{v+(i,\ldots,i)}$. Also if $v=(v_1,\ldots,v_n)$, we put $\mathcal{O}(v)=\mathcal{O}(v_1)\boxtimes\cdots\boxtimes\mathcal{O}(v_n)\in\coh X$ for $X=\PP^1\times\cdots\times\PP^1$.

We first prepare the following.
\begin{Lem}\label{generator}
Put $m=\min\{v_i\mid 1\leq i\leq n\}$. Then $M_v$ is generated in degree $-m$.
\end{Lem}
\begin{proof}
	Certainly we may assume $v_1,\ldots,v_{n-1}\geq0$ and $v_n=0$. Writing $T=k[x,y]$ with $\deg x=\deg y=1$, we have $M_v=T(v_1)\seg\cdots\seg T(v_{n-1})\seg T$. Consider the maps $T^{\oplus(v_i+1)}\to T(v_i)\to C_i\to 0$ for $1\leq i\leq n$, where the first map is the canonical one and $C_i$ the cokernel. Since $v_n=0$ we have $C_n=0$. Noting that taking the Segre product is exact, we get a map $R^\oplus \to M_v$ whose cokernel is $C_1\seg\cdots\seg C_{n-1}\seg C_n=0$ since $C_n=0$. Therefore $M_v$ is generated in degree $0$.
\end{proof}

\begin{Lem}\label{M_v in sg}
If $0<v<1$, then $M_v\in\Db(\mod^{\geq0}\!R)\cap\Db(\mod^{>0}\!R)^\ast$, the fundamental domain for $\sg^\Z\!R$.
\end{Lem}
\begin{proof}
	If $v<1$, then $\min\{v_i\mid 1\leq i\leq n\}\leq0$, so \ref{generator} shows that $M_v\in\mod^{\geq0}\!R$. Next, note that $M_v$ is CM if $0\leq v\leq 1$ (see e.g.\! \cite{GW1,St,VdB91}), thus $\RHom_R(M_v,R)=\Hom_R(M_v,R)=M_{-v}$. By \ref{generator} again, this is concentrated in degree $\geq1$. We obtain the desired assertion from these two claims.
\end{proof}


Now we are ready to prove the main result of this section.
\begin{proof}[Proof of \ref{P^1}]
	We have a standard tilting object $\bigoplus_{0\leq v\leq 1}\mathcal{O}(v)\in\Db(\coh X)$ obtained as the tensor product of the ones $\mathcal{O}\oplus\mathcal{O}(1)\in\Db(\coh\PP^1)$. Note that this is the image of $U:=\bigoplus_{0\leq v\leq1}M_v\in\qmod^\Z\!R$ under the equivalence $\qmod^\Z\!R\xsimeq\coh X$,
	Now consider the semi-orthogonal decomposition
	\[ \Db(\qmod^\Z\!R)=\Kb(\proj^{-1}\!R)\perp\sg^\Z\!R\perp\Kb(\proj^{0}\!R) \]
	from \ref{ssodd}. We have seen in \ref{M_v in sg} that the summand $T=\bigoplus_{0<v<1}M_v$ of $U$ lies in $\sg^\Z\!R$, and clearly $R\in\Kb(\proj^0\!R)$ and $R(1)\in\Kb(\proj^{-1}\!R)$ for the remaining summands. We therefore see that $T\in\sg^\Z\!R$ is a tilting object.
	
	Now we show that the endomorphism ring $\End_{\sg R}^\Z(T)$ being $(n-2)$-representation infinite.
	The endomorphism ring of $U$ is $B:=\End_X(\bigoplus_{0\leq v\leq 1}\mathcal{O}(v))=\End_{\PP^1}(\mathcal{O}\oplus\mathcal{O}(1))^{\otimes n}$, the $n$-fold tensor product of the path algebra of the Kronecker quiver $\xymatrix{\circ\ar@2[r]&\circ}$, thus it is $n$-representation infinite by \cite[2.10]{HIO}.
	Notice that $\End_{\sg R}^\Z(T)$ is the quotient of $B$ by the idempotents at $2$ vertices, the one $e_1$ at the sink (corresponding to $M_1=R(1)$) and the one $e_0$ at the source (corresponding to $M_0=R$). It is easy to see that the injective dimension of the simple $B$-module corresponding to the summand $M_v$ is $n-\sum_{i=1}^nv_i$. Therefore, applying \ref{-1RI} to $B$ and the idempotent $e_0$ shows that $B/(e_0)$ is $(n-1)$-representation infinite. Next, we can similarly apply \ref{-1RI*} to $B/(e_0)$ and $e_1$ to deduce that $\End_{\sg R}^\Z(T)=B/(e_0,e_1)$ is $(n-2)$-representation infinite.
	
	The existence of the last commutative diagram follows from \cite[0.1]{HaI}.
\end{proof}

\subsection{Proof of \ref{n=3}}

\begin{proof}[Proof of \ref{n=3}]
	(1)  This is valid in any dimension. Since $R=k[x_1,y_1]\seg\cdots\seg k[x_n,y_n]$ is the Segre product of rings of the common $a$-invariant $-2$, $R$ also has $a$-invariant $-2$. Noting that $\dim R=n+1$, the graded Auslander-Reiten duality (see e.g.\! \cite[3.5]{IT}\cite[3.19]{HaI}) implies that the category $\sg^\Z\!R$ has a Serre functor $\nu=(-2)[n]$. Therefore, $\nu_{n-2}^{-1}=(2)[-2]=\Om^2(2)$ has a square root $(1)[-1]=\Om(1)$.
	
	(2)  We only prove $\Om M_{100}(1)=M_{011}$. Let $S_i=k[x_i,y_i]$ for $i=1,2,3$ with $\deg x_i=\deg y_i=1$. Consider the Koszul complex $0\to S_1(-1)\to S^{\oplus2}\to S(1)\to k(1)\to 0$. Taking the Segre product with $S_2\seg S_3$, we get an exact sequence
	\[ \xymatrix{ 0\ar[r]&M_{-100}\ar[r]& R^{\oplus2}\ar[r]&M_{100}\ar[r]& 0 }. \]
	Noting that $M_{-100}(1)=M_{011}$, this sequence shows $\Om M_{100}(1)=M_{011}$.
	
	(3)(4)  The first assertion of (3) follows from \ref{tilt}. Indeed, the pair $(U,P)$ with $P=\Hom_{\sg R}^\Z(T,T_0)$ and $U=\Hom_{\sg R}^\Z(T,\Om T(1))$ is a strict root pair.
	We next verify the second assertion of (3). By the proof of \ref{P^1}, the algebra $\End_{\sg R}^\Z(T)$ is the quotient of $\End_{\PP^1}(\mathcal{O}\oplus\mathcal{O}(1))^{\otimes 3}$, which is presented by the cube below, by the idempotents corresponding to the sink and the source.
	\[ \xymatrix{
		&M_{100}\ar@2[r]\ar@2[dr]&M_{110}\ar@2[dr]&\\
		M_{000}\ar@2[ur]\ar@2[r]\ar@2[dr]&M_{010}\ar@2[dr]\ar@2[ur]&M_{101}\ar@2[r]&M_{111}\\
		&M_{001}\ar@2[r]\ar@2[ur]&M_{011}\ar@2[ur]& } \]
	Then we obtain the quiver $Q$ by deleting these vertices, which proves (4).
	
	(5)  This is a consequece of \ref{P^1}.
	
	(6)  We start from the subcategories in the second row. By \cite[5.2]{Ha7} these are $3$-cluster tilting subcategories of respective categories in (5). Noting that the functor $-\lotimes_{kQ}U=\tau^{-1/2}$ on $\Db(\mod kQ)$ corresponds to $\Om(1)=(1)[-1]$ on $\sg^\Z\!R$, we immediately get the corresponding subcategories on the first row. In particuar, they are $3$-cluster tilting subcategories.
\end{proof}
	
%
%
%
%

\end{document}